\documentclass[11pt]{amsart}
\usepackage{amsfonts, amssymb, latexsym}
\usepackage{tikz, multicol,parcolumns, xcolor}
\usepackage{hyperref}
\usepackage{float}
\usepackage{wrapfig}
\usetikzlibrary{calc,arrows}

\restylefloat{figure}

\newtheorem{theorem}{Theorem}[section]
\newtheorem{proposition}[theorem]{Proposition}
\newtheorem{lemma}[theorem]{Lemma}

\newtheorem{claim}[theorem]{Claim}
\newtheorem*{claim*}{Claim}
\newtheorem{corollary}[theorem]{Corollary}

\newtheorem{Main Conjecture}[theorem]{Main Conjecture}

\newtheorem{definition}[theorem]{Definition}
\theoremstyle{remark}

\newtheorem{example}[theorem]{Example}

\newtheorem{remark}[theorem]{Remark}
\theoremstyle{plain}

\newcommand{\bem}{\mathcal{BEM}^{Y_0,Q}}

\newcommand{\lra}[1]{\mathrel{\mathop{\longrightarrow}^{\mathrm{#1}}}} 
\newcommand{\lmp}[1]{\mathrel{\mathop{\longmapsto}^{\mathrm{#1}}}} 

\newcommand{\caS}{\mathcal{S}}

\newcommand{\frt}{\mathfrak{t}}

\newcommand{\field}{\mathbb}
\newcommand{\C}{{\field C}}
\newcommand{\R}{{\field R}}
\newcommand{\Z}{{\field Z}}
\newcommand{\N}{{\field N}}
\newcommand{\Q}{{\field Q}}

\newcommand{\e}{\epsilon}
\newcommand{\ga}{\alpha}
\newcommand{\gb}{\beta}

\newcommand{\excise}[1]{}

\begin{document}
\pagestyle{plain}

\mbox{}
\title{${\sf K}$-orbit closures and Barbasch-Evens-Magyar varieties}
\author{Laura Escobar}
\address{Laura Escobar \\ Department of Mathematics and Statistics \\ Washington University in St.~Louis \\ St.~Louis, MO 63130 \\ USA}
\email{laurae@wustl.edu}

\author{Benjamin J. Wyser}
\address{Benjamin J. Wyser}
\email{ben.wyser@gmail.com}

\author{Alexander Yong}
\address{Alexander Yong \\ Department of Mathematics \\ University of Illinois at Urbana--Champaign \\ Urbana, IL 61801 \\ USA}
\email{ayong@uiuc.edu}
\date{May 16, 2022}

\maketitle

\begin{abstract}
We define the \emph{Barbasch-Evens-Magyar varieties}. We show they are 
isomorphic to the smooth varieties defined in [D.~Barbasch-S.~Evens '94] that 
map generically finitely to symmetric orbit closures, thereby giving resolutions of singularities in certain cases. Our definition parallels [P.~Magyar '98]'s construction of the \emph{Bott-Samelson varieties} [H.~C.~Hansen '73, M.~Demazure '74]. From this alternative viewpoint, one deduces a
graphical description in type $A$, stratification into closed subvarieties of the same kind, and determination of the torus-fixed points. 
Moreover, we explain how these manifolds inherit a natural symplectic 
structure with Hamiltonian torus action. We then express
the moment polytope in terms of the moment polytope of a Bott-Samelson variety. 
\end{abstract}

\tableofcontents

\section{Introduction}

Let $X$ be a generalized flag variety of the form ${\sf G}/{\sf B}$,
where ${\sf G}$ is a connected reductive complex algebraic group and ${\sf B}$ is a Borel subgroup of ${\sf G}$. 
The left action of ${\sf B}$ on $X$ has finitely many orbits ${\sf B}w{\sf B}/{\sf B}$, where $w$ is a Weyl group element. The {\bf Schubert variety} $X_w$ is the closure ${\overline{{\sf B}w{\sf B}/{\sf B}}}$ of the ${\sf B}$-orbit. The study of Schubert variety singularities is of interest; see, e.g., \cite{BilleyLakshmibai, Brionnotes, BilleyJapan} and the references therein. 

In the 1970s,
H.C.~Hansen \cite{Hansen} and M.~Demazure
\cite{Demazure} constructed a \emph{Bott-Samelson variety} 
$BS^{Q}$ for each reduced word $Q$ of $w$, building on ideas
of R.~Bott-H.~Samelson \cite{BottSamelson}. These manifolds are 
resolutions of singularities of $X_w$. In recent years, Bott-Samelson varieties have been used, e.g., in studies of Schubert calculus
(M.~Willems \cite{Willems}), Kazhdan-Lusztig polynomials (B.~Jones-A.~Woo \cite{Jones-Woo}), Standard Monomial Theory
(V.~Lakshmibai-P.~Littelmann-P.~Magyar \cite{LakLitMag}), Newton-Okounkov bodies (M.~Harada-J.~Yang \cite{Harada-Yang}),
and matroids over valuation rings (A.~Fink-L.~Moci \cite{FinkMoci}).

In 1983, A.~Zelevinsky \cite{Zelevinsky} gave a different resolution for Grassmannian Schubert varieties, presented as
configuration spaces of vector spaces prescribed by dimension and containment conditions.
In 1998, P.~Magyar \cite{Magyar} gave a new description of  $BS^Q$ in the same spirit, replacing the quotient by group action definition with an iterated fiber product. 

Similar constructions have been used subsequently in, e.g.,
\begin{enumerate} 
\item[(I)] P.~Polo's proof that every polynomial $f\in 1+q{\mathbb Z}_{\geq 0}[q]$ is a Kazhdan-Lusztig polynomial (in type $A$) \cite{Polo};
\item[(II)] A.~Cortez's proof of the singular locus theorem for Schubert varieties in type $A$ \cite{Cortez} (cf.~\cite{Manivelsing, BilleyWarrington, Kassel}); 
\item[(III)] N.~Perrin's extension of Zelevinsky's resolution to minuscule
Schubert varieties \cite{Perrin1} (one application is \cite{Perrin2});
\item[(IV)] A.~Woo's classification of ``short'' Kazhdan-Lusztig polynomials \cite{Wooshort}; 
\item[(V)] the definition of the \emph{brick variety}, which provides resolutions of singularities of Richardson varieties \cite{Escobar}; and  
\item[(VI)] the connection \cite{EPTY-16+} of Magyar's definition to S.~Elnitsky's rhombic tilings \cite{Elnitsky}. 
\end{enumerate}

We are interested in the parallel story where orbit closures for symmetric subgroups
replace Schubert varieties.
A {\bf symmetric subgroup} ${\sf K}$ of ${\sf G}$ is a group comprised of the fixed points ${\sf G}^\theta$ of an involution $\theta$ of ${\sf G}$.  Like ${\sf B}$, ${\sf K}$ is {\bf spherical}, meaning that it has finitely many orbits ${\mathcal O}$ under the left action on $X$. The study of the singularities of a ${\sf K}$-orbit closure $Y={\overline{\mathcal O}}$ is relevant to the theory of
Kazhdan-Lusztig-Vogan polynomials and Harish-Chandra modules for a certain real Lie group ${\sf G}_{\R}$. This may be compared
with the connection of Schubert varieties to Kazhdan-Lusztig polynomials and the representation theory of complex semisimple Lie algebras. 

In 1994, D.~Barbasch-S.~Evens \cite{Barbasch-Evens-94} constructed a smooth variety, using a quotient description that extends the one for Bott-Samelsons from \cite{Demazure,Hansen}.  This variety comes equipped with a natural map to a particular ${\sf K}^0$-orbit closure, where ${\sf K}^0$ is the connected component of $1$.
In certain situations, this map provides a resolution of singularities of the orbit closure in question.

This paper introduces and initiates our study of the
\emph{Barbasch-Evens-Magyar variety} (BEM variety).
 Just as  P.~Magyar
\cite{Magyar} describes, \emph{via} a fiber product, a variety that is equivariantly isomorphic to a Bott-Samelson variety, the BEM variety reconstructs the manifold of \cite{Barbasch-Evens-94} (Theorem~\ref{thm:main}(I)). 

Our definition gives \emph{general type} results
about the varieties of  \cite{Barbasch-Evens-94}:
\begin{itemize}
\item a stratification (in the sense of 
\cite[Section~1.1.2]{Knutson}) into smaller BEM varieties (Corollary~\ref{cor:stratified});
\item description of its torus fixed points (Proposition~\ref{prop:thefixed});  
\item a symplectic structure with Hamiltonian torus action as well as analysis of the moment map image, e.g. we show that it is the convex hull of certain Weyl group reflections of the moment polytope of a Bott-Samelson variety (Theorem~\ref{thm:moment}); and
\item an analogue of the brick variety (Theorem~\ref{thm:main}(II)).
\end{itemize}
In type $A$ we give a diagrammatic description of the BEM varieties (Section~\ref{sec:A-BEMs}) in linear algebraic terms, avoiding the
algebraic group generalities. For example, we obtain more specific results (Section~\ref{sec:furtherPQ})
for ${\sf K}={\sf GL}_p\times {\sf GL}_q$ acting on
${\sf GL}_{p+q}/{\sf B}$. We show (Theorem~\ref{prop:theone}) that the study of BEM polytopes can be reduced to a special case. 
We determine the torus weights for this special case (Theorem~\ref{prop:tweights}) which permits us to partially
understand the vertices (Corollary~\ref{cor:vertexpartial}). We also give a combinatorial characterization of the dimension of the BEM polytope (Theorem~\ref{thm:dimaug2}).

We anticipate that many uses of the
Zelevinsky/Magyar-type construction of the Bott-Samelson variety, such as 
(I)--(VI) above, have BEM versions. In particular, an analogue of (II), even in the case of ${\sf K}={\sf GL}_p\times {\sf GL}_q$, would bring important new
information about the singularities of the symmetric orbit closures. More generally, analogues of (I)--(IV)
would illuminate the combinatorics of the celebrated Kazhdan-Lusztig-Vogan polynomials.

\section{Background on ${\sf K}$-orbits}\label{sec:background}

In this section we describe the background in general. 
See Section~\ref{sec:A-BEMs} for background on $\sf K$-orbits of type A.

Let ${\sf G}$ be a connected complex reductive algebraic group and ${\sf B}$ a Borel subgroup of ${\sf G}$ containing
a maximal torus ${\sf T}$. 
Furthermore we assume $\theta$ is an involution  of ${\sf G}$ and that ${\sf B}$ and ${\sf T}$ are $\theta$-stable.
We denote by ${\sf K}$ the symmetric subgroup ${\sf G}^\theta$.
Throughout this paper we assume that $\sf K$ is the connected component of the fixed point set of $\theta$.

Let ${\sf W}=N_{\sf G}({\sf T})/{\sf T}$
be the Weyl group.  
Let $r$ be the rank of the root system of ${\sf G}$ and $\Delta=\{\ga_1,\hdots,\ga_r\}$ 
be the system of simple roots corresponding to ${\sf B}$, with 
$\{\omega_1,\hdots,\omega_r\}$
the corresponding fundamental weights.  Denote the simple reflection corresponding to the simple root $\ga_i$ by $s_{i}$.
Thus, ${\sf W}$ is generated by the simple reflections 
$\{s_{i} \mid 1\leq i\leq r\}$.

Given $I\subseteq \Delta$, ${\sf P}_I$ is the standard parabolic subgroup
of ${\sf G}$ corresponding to $I$; namely, 
\begin{equation}
\label{eqn:standardpara}
{\sf P}_I = {\sf B} \cup \left( \bigcup_{w\in W_I} {\sf B} w {\sf B} \right),
\end{equation}
where $W_I$ is the set of $w\in W$ such that $w=s_{i_1}\cdots s_{i_k}$ where all $i_j\in I$.
$P_I$ is a {\bf minimal parabolic} if $I=\{\ga_i\}$; it is a {\bf maximal parabolic} if
$I=\{\ga_1,\hdots,\widehat{\ga_i},\hdots,\ga_r\}$. These are denoted 
${\sf P}_i$ and ${\sf P}_{\widehat{i}}$, respectively.

As described in \cite[Section 3.10]{Richardson-Springer-90}, the \textbf{Richardson-Springer monoid} ${\mathcal M}({\sf W})$ is generated by the simple reflections $s_i$ of ${\sf W}$, with relations $s_i^2 = s_i$, together with the braid relations of $\sf W$.  As a set, this monoid may be canonically identified with 
${\sf W}$, with the ordinary product on ${\sf W}$ being replaced by the \textbf{Demazure product} $\star$, a product having the property that 
\[ s_i \star w = \begin{cases}
	s_iw & \text{ if $\ell(s_iw) > \ell(w)$,}\\
	w & \text{ otherwise,} \end{cases}
\]
where $\ell(\cdot)$ denotes ordinary Coxeter length, and where the juxtaposition $s_iw$ denotes the ordinary product in ${\sf W}$.
A \textbf{word}  $Q=(j_1,j_2,\hdots,j_N)$ is an ordered tuple of numbers from 
$\{1,2,\ldots, r\}$.  
Let 
$\text{Dem}(Q):=s_{j_1} \star s_{j_2} \star \hdots
\star s_{j_N}$.
 If $\text{Dem}(Q) = w$, then 
$Q$ is a \textbf{Demazure word} for $w$.

Consider the natural projection
$\pi_i: {\sf G}/{\sf B} \rightarrow {\sf G}/{\sf P}_i$.
Given a ${\sf K}$-orbit closure $Y$ on ${\sf G}/{\sf B}$ and a simple reflection $s_{i} \in {\sf W}$, 
$s_{i} \star Y:=\pi_i^{-1}(\pi_i(Y))$
is also a ${\sf K}$-orbit closure.
This extends to an ${\mathcal M}({\sf W})$-action on the set of ${\sf K}$-orbit closures:  given a Demazure word
$Q=(s_{j_1},\hdots,s_{j_N})$ for $w$, define
\[w \star Y := s_{j_1} \star (s_{j_2} \star \hdots \star (s_{j_N} \star Y) \hdots ).\]
The ${\sf K}$-orbit closure $w \star Y$ is independent of the choice of Demazure word $Q$ for $w$ \cite[Section 4.7]{Richardson-Springer-90}.  

The
\textbf{weak order} on the set of ${\sf K}$-orbit closures is defined by 
\[Y \leq Y' \iff Y' = w \star Y\]
for some $w \in {\mathcal M}({\sf W})$.
The minimal elements of this order are the {\bf closed orbits}, i.e., those $Y_0={\mathcal O}={\overline {\mathcal O}}$.
The following is well-known; see, e.g., \cite[Proposition~2.2(i)]{Brion1999}:

\begin{lemma}
\label{lemma:closedisflag}
Each closed orbit is isomorphic to ${\sf K}/{\sf B}'$ where ${\sf B}'$ is a Borel subgroup of ${\sf K}$. In particular, every closed orbit is smooth.
\end{lemma}

\begin{remark}
If $\sf K$ is disconnected, then \cite[Proposition~2.2(i)]{Brion1999} says that $Y_0$ is isomorphic to a finite union of flag manifolds ${\sf K}^0/{\sf B}'$, where ${\sf K}^0$ is the connected component of $1$ and ${\sf B}'$ is a Borel subgroup of ${\sf K}^0$.
\end{remark}

\section{Barbasch-Evens-Magyar varieties}\label{sec:BEManyG}

As in the previous section, ${\sf G}$ is a connected reductive complex algebraic group, ${\sf K}={\sf G}^\theta$ is connected, and ${\sf B}$ is a $\theta$-stable Borel subgroup of ${\sf G}$. 
We begin with the 
definition of the manifold of D.~Barbasch-S.~Evens
\cite[Section~6]{Barbasch-Evens-94}.

If ${\sf B}^{k-1}$ acts on $X_k\times\cdots\times X_1$ by  
\begin{equation}
\label{eqn:recallabc345}
(b_k,\ldots,b_1)\cdot(x_k,\cdots,x_1)=(x_kb_k,b_k^{-1}x_{k-1}b_{k-1},\ldots, b_{2}^{-1}p_1),
\end{equation}
then $X_k \times^{\sf B}\cdots\times^{\sf B} X_1$ denotes the quotient of $X_k\times\cdots\times X_1$ by this action, when it exists.
In \cite[\S 3.3]{Graham-Zierau} this quotient is shown to exist in cases which include when the $X_i$ for $i>1$ are subgroups of $\sf G$ such that ${\sf B}\subset X_i\cap X_{i+1}$ and $X_1={\sf P}/{\sf B}$ for some parabolic subgroup $\sf P$ of $\sf G$.

Let $Y_0$ be a closed $K$-orbit and $Q=(j_1,j_2,\ldots,j_N)$ a (not necessarily reduced) word.
The \textbf{Barbasch-Evens variety}  \cite[(6.3.5)]{Barbasch-Evens-94} for $Y_0$ and $Q$ is
	\begin{equation}
	\label{eq:BEvar}
	 {\mathcal B}{\mathcal E}^{Y_0,Q} := \widetilde{Y_0} \times^{\sf B} 	{\sf P}_{j_N} \times^{\sf B} {\sf P}_{j_{N-1}} \times^{\sf B} \hdots \times^{\sf B} {\sf P}_{j_1}/{\sf B},	
	\end{equation}
where $\widetilde{Y_0}$ denotes the preimage of $Y_0$ in ${\sf G}$ under ${\sf G} \rightarrow {\sf G}/{\sf B}$.
By \cite[\S 3.3]{Graham-Zierau} this quotient exists.
Recall that by Lemma~\ref{lemma:closedisflag} $Y_0$ is smooth.
Since ${\mathcal B}{\mathcal E}^{Y_0,Q}$ is an iterated ${\mathbb P}^1$-bundle over $Y_0$, then ${\mathcal B}{\mathcal E}^{Y_0,Q}$ is a manifold.

\begin{remark}
Even though \eqref{eq:BEvar} looks different from  \cite[(6.3.5)]{Barbasch-Evens-94} \eqref{eq:BEvar} is actually a special case.
This is because by \cite[(6.3.2)]{Barbasch-Evens-94} any closed $\sf K$-orbit $Y_0$ is isomorphic to a ${\sf K}\times_{{\sf K}\cap {\sf P}}{\sf P}/{\sf B}$ so the preimage of $Y_0$ in $\sf G$ can be taken to be ${\sf K}\times_{{\sf K}\cap {\sf P}}{\sf P}$. 
In general, ${\sf K}\times_{{\sf K}\cap {\sf P}} {\sf P}/{\sf B}$ need not be a closed $\sf K$-orbit.
Description \eqref{eq:BEvar} appears in \cite[Section 5]{Knutson-12}. 
\end{remark}

${\sf K}$ acts on ${\mathcal B}{\mathcal E}^{Y_0,Q}$ by
	\begin{equation}
	\label{eqn:actionMay16}
	{\sf k}\cdot[{\sf g},p_N,\hdots,p_1{\sf B}]=[{\sf k}{\sf g},p_N,\hdots,p_1{\sf B}].
	\end{equation}

There is a ${\sf K}$-equivariant map $\beta:{\mathcal B}{\mathcal E}^{Y_0,Q} \rightarrow G/B$ given by 
\begin{equation}
\label{eqn:themapMay16}
\displaystyle{[{\sf g},p_N,\hdots,p_1{\sf B}] \lmp{\beta} {\sf g}p_N \hdots p_1{\sf B}}.
\end{equation}
Indeed, both the action (\ref{eqn:actionMay16}) and the map
(\ref{eqn:themapMay16}) are well-defined, i.e., independent of choice of
representative of the equivalence class $[{\sf g},p_N,\hdots,p_1]$.

R.~W.~Richardson-T.~A.~Springer \cite{Richardson-Springer-90} proved
that for any $Y$, there is a
closed orbit $Y_0$ (possibly
non-unique) below it in weak order.  
That is, there is some
$w \in {\sf W}$
such that 
$Y = w \star Y_0$ and $\dim(Y)=\ell(w)+\dim(Y_0)$.
Let $Y$ and $w$ be as above and $Q=({j_1},{j_2},\ldots,{j_{\ell(w)}})$ be a reduced word for $w$.
Then $\beta:{\mathcal B}{\mathcal E}^{Y_0,Q} \rightarrow Y$ is generically finite since $\beta$ is
surjective onto $Y$ (by \cite[Proposition~6.4]{Barbasch-Evens-94}) and since
\[\dim({\mathcal B}{\mathcal E}^{Y_0,Q})=\dim(Y_0)+\ell(w)=\dim(Y).\]
When $({\sf G},{\sf K})=({\sf GL}_{p+q},{\sf GL}_p \times {\sf GL}_q)$, $\beta$ is a resolution of singularities for $Y$, again by \cite[Proposition~6.4]{Barbasch-Evens-94} (see also \cite[Lemma~5.1]{Knutson-12}).
Even for this case,
not all reduced words $Q$ give a resolution of singularities
$\beta$, although this is true if $\beta$ is generically one-to-one, that is, under the condition that $\dim(Y)=\ell(w)+\dim(Y_0)$. In general, the image of $\beta$ is $Y$ and $\beta:{\mathcal B}{\mathcal E}^{Y_0,Q} \rightarrow Y$ is a resolution of singularities for $Y$	
	 under certain hypotheses \cite[Section~6]{Barbasch-Evens-94}.\footnote{To construct a resolution of singularities, it is not necessary to take $Y_0$ to be a closed orbit.
	We need only take $Y_0$ to be a smooth orbit closure underneath $Y$ in weak order \cite{Knutson-12}, or take $Y_0$ to be the closure of a ``distinguished'' orbit \cite{Barbasch-Evens-94}.  	
	However, closed orbits are both smooth and distinguished. 
	Taking them as a starting point seems closest in spirit to the construction of the Bott-Samelson resolution.}
These hypotheses are sufficiently technical that we do not wish to recall them here since we will
not use them. 

We now define the Barbasch-Evens-Magyar varieties.
See Section~\ref{sec:A-BEMs} for a diagrammatic description in type A.

\begin{definition}[Barbasch-Evens-Magyar variety]
\label{def:BEM}
Suppose that 
$Q=({j_1},{j_2},\ldots,{j_N})$
is a (not necessarily reduced) word. 
The Barbasch-Evens-Magyar variety is
\begin{equation}
\label{eqn:Magyarform}
\bem:=Y_0
\times_{{\sf G}/{\sf P}_{j_{N}}}{\sf G}/{\sf B}\times_{{\sf G}/{\sf P}_{j_{N-1}}}
\cdots\times_{{\sf G}/{\sf P}_{j_1}}{\sf G}/{\sf B}.
\end{equation}
\end{definition}
Recall that if $X_1 \stackrel{f}{\rightarrow} Y$ and $X_2 \stackrel{g}{\rightarrow} Y$ are two varieties mapping to the same variety $Y$, then
\begin{equation}
\label{eqn:Magyarhelper}
X_1\times_{Y}X_2=\{(x_1,x_2)\in X_1\times X_2 \mid f(x_1)=g(x_2)\}
\end{equation}
denotes the {\bf fiber product}. In (\ref{eqn:Magyarform}),
each map of (\ref{eqn:Magyarhelper}) is the natural projection
${\sf G}/{\sf B}\to {\sf G}/{\sf P}_{j_i}$ defined by $g{\sf B} \mapsto g{\sf P}_{j_i}$ 
(or, in the case of $Y_0$, the restriction of said projection).

Evidently, ${\sf K}$ acts diagonally on $\bem$. 
Our next theorem asserts that the projection
\begin{align}
\label{eqn:theta987}
\varphi: \bem&\rightarrow {\sf G}/{\sf B}\\
(x_{N+1},x_N,\ldots,x_1)&\mapsto x_1\nonumber
\end{align}
maps into $Y$.
We remark that since we are not assuming that $Q$ is a reduced word, this map may not be generically one-to-one.

\begin{theorem}
\label{thm:main}
Suppose that $Y=w\star Y_0$
for a closed orbit $Y_0$ and $Q$ is a (not necessarily reduced) Demazure word for $w$.
\begin{itemize}
\item[(I)] 
$\bem\cong \mathcal{BE}^{Y_0,Q}$
as ${\sf K}$-varieties.
\item[(II)] Suppose $Y$ is the closure of the {\sf K}-orbit ${{\sf K}g{\sf B}}$. The fiber of $\varphi$ over a point 
of ${\sf K}g{\sf B}$ of $Y$ is smooth of dimension
$\dim(\bem)-\dim(Y)$.
\end{itemize}
\end{theorem}
\begin{proof}
We prove (I) by a modification of the argument of P.~Magyar in the Schubert setting. The map
\begin{align} \phi: {\mathcal B}{\mathcal E}^{Y_0,Q} &\rightarrow Y_0\times ({\sf G}/{\sf B})^N  \\
[g,p_N,p_{N-1},\ldots,p_1{\sf B}]&\mapsto (g{\sf B},gp_N {\sf B}, gp_N p_{N-1} {\sf B},\cdots, gp_N p_{N-1}\cdots p_1 {\sf B}), 
\end{align}
is well defined (independent of choice of representative), ${\sf K}$-equivariant, and 
$\phi({\mathcal B}{\mathcal E}^{Y_0,Q})\subseteq \bem$
since $p_i\in {\sf P}_{j_i}$. 

\noindent
{\sf $\phi$ is injective:} If 
$\phi([g,p_N,p_{N-1},\ldots,p_1{\sf B}])=\phi([g',p'_N,p'_{N-1},\ldots,p'_1{\sf B}])$,
then there exist $b_0,b_1,\ldots,b_N\in{\sf B}$ such that
$$g=g'b_0, \ \ \ gp_N=g'p'_N b_N, \ \ \ \ldots, \ \ \ gp_N\cdots p_1=g'p'_N \cdots p'_1b_1.$$
Combining these equations with the definition of ${\mathcal B}{\mathcal E}^{Y_0,Q}$ (specifically (\ref{eqn:recallabc345})),
\begin{align*}[g,p_N,p_{N-1},\ldots,p_1{\sf B}]&=[g'b_0,b_0^{-1}p'_N b_N, b_N^{-1}p'_{N-1}b_{N-1},\ldots, b^{-1}_2p'_1b_1{\sf B}]\\
&=[g',p_N',p_{N-1}',\ldots,p_1'{\sf B}],\end{align*}
establishing injectivity.

\noindent
{\sf $\phi$ is surjective:} Let 
$(g {\sf B},g_N {\sf B},g_{N-1} {\sf B},\ldots,g_1 {\sf B}) \in \bem$.

\begin{claim}
$[g,g^{-1}g_N,g_N^{-1}g_{N-1},\ldots,g_2^{-1}g_1{\sf B}]\in {\mathcal B}{\mathcal E}^{Y_0,Q}$. 
\end{claim}
\noindent
\emph{Proof of Claim:} First, by definition $g\in \widetilde{Y_0}$, as desired.
Second, by (\ref{eqn:Magyarform}) and (\ref{eqn:Magyarhelper}) combined we have
\[g{\sf P}_{j_N}=g_N {\sf P}_{j_N} \implies g^{-1}g_N\in {\sf P}_{j_N}.\]
Similarly, in general 
\[g_i{\sf P}_{j_{i+1}}=g_{i+1}{\sf P}_{j_{i+1}} \implies g_i^{-1}g_{i+1}\in {\sf P}_{j_{i+1}},\]
as required. \qed

Combining the claim with
$\phi(g,g^{-1}g_N,g_N^{-1}g_{N-1},\ldots,g_2^{-1}g_1{\sf B})=(g {\sf B}, g_N {\sf B},g_{N-1} {\sf B},\ldots,g_1 {\sf B})$,
we obtain
$\phi({\mathcal B}{\mathcal E}^{Y_0,Q})=\bem$.

\noindent
{\sf $\varphi$ maps into $Y$:}
Since $\beta$ maps into $Y$,
$$\beta\circ\phi^{-1}(g {\sf B},g_N {\sf B},g_{N-1} {\sf B},\ldots,g_1 {\sf B})= \beta(g,g^{-1}g_N,g_N^{-1}g_{N-1},\ldots,g_2^{-1}g_1{\sf B})= g_1 {\sf B} \in Y.$$
However, by definition $\varphi(g {\sf B},p_N {\sf B},p_{N-1} {\sf B},\ldots,p_1 {\sf B})=p_1{\sf B}$
and so $\varphi$ maps into $Y$ as well.

Since $\bem$ is smooth (and thus normal) and
${\mathcal B}{\mathcal E}^{Y_0,Q}$ is irreducible, 
the bijective morphism (of ${\mathbb C}$-varieties) above is an isomorphism of varieties by Zariski's main theorem (see, e.g., \cite[Theorem 5.2.8]{Spring}).

For (II), we apply:
\begin{theorem}\cite[Corollary 10.7 of chapter III]{Hartshorne} 
\label{thm:Hartshorne}
Let 
$f:X\rightarrow Y$
be a morphism of varieties over an algebraically closed field $k$ of characteristic 0, and assume that $X$ is nonsingular. There is a nonempty open subset $V\subset Y$ such that 
$f:f^{-1}(V)\rightarrow V$
is smooth. In the case in which 
$f^{-1}(V)\neq \emptyset$, 
the fiber $f^{-1}(v)$ is nonsingular and 
$\dim(f^{-1}(v))=\dim(X)-\dim(Y)$
for all $v\in V$. 
\end{theorem}

Let  $f$ be the projection map $\varphi: \bem\to Y$.
Since $\bem$ is nonsingular, by Theorem~\ref{thm:Hartshorne} applied to this 
$f$, there exists a nonempty 
$V\subset Y$
such that $f$ restricted to $f^{-1}(V)$ is smooth. If $v\in V$ then said theorem 
says $\dim(f^{-1}(v))$ is of the desired dimension.

However, we want the above to be true for $p\in {\sf K}g{\sf B}$. To see this, note that everything said above holds for 
$f^{-1}({\sf k}V)$ for all ${\sf k}\in {\sf K}$ since $f$ is ${\sf K}$-equivariant and multiplication by ${\sf k}$ is a smooth morphism. 
Let $p \in{\sf K}g{\sf B}$ 
be a general point. Since ${\sf K}p {\sf B}$ is dense in 
$Y$, $Y\cap {\sf k}V\neq\emptyset$ for all ${\sf k}\in {\sf K}$. Now we can pick ${\sf k}$
so that $p \in {\sf k}V$, completing the argument.
\end{proof}

The generic fibers of part (II) of the theorem may be considered
an analogue of the \emph{brick variety} of 
\cite{Escobar}, which is the generic fiber of the Bott-Samelson map (this generic fiber being positive dimensional only when $Q$ is not a reduced word). See \emph{loc.~cit.} for a connection to the \emph{brick polytope} of \cite{Pilaud-Santos,Pilaud-Stump} and the associahedron.

\begin{remark}
\label{prop:P1bundle}
$\bem$ is an iterated ${\mathbb P}^1$-bundle over $Y_0$.
\end{remark}

Let 
\[p_{Y_0,Q}(z)=\sum_{k\geq 0}z^k\dim_{\Q}
H^{2k}(\bem; {\Q}),\]
and
\[r_{Y_0}(z)=\sum_{k\geq 0}z^k \dim_{\Q}H^{2k}(Y_0;{\Q})\]
be the Poincar\'{e} polynomials of $\bem$ and $Y_0$, respectively.

\begin{corollary}
\label{cor: poincare}
$p_{Y,Q}(z)=r_{Y_0}(z)(1+z)^N$, where $r_{Y_0}(z)$ is known since each closed orbit $Y_0$ is isomorphic to the flag variety of ${\sf K}$.
\end{corollary}
\begin{proof}
In view of Remark~\ref{prop:P1bundle}, the claim 
follows by repeated applications of the Leray-Hirsch theorem.
\end{proof}

Following \cite[Section~1.1.2]{Knutson-12},
a \textbf{stratification by closed subvarieties} of a variety $X$ is a decomposition 
$X=\bigcup_\xi \caS_\xi$ 
into closed varieties $\caS_{\xi}$ such that the intersection of any two closed strata is the union of strata.
We have a stratification of $\bem$ with strata given by subwords $P$ of $Q$. 
A \textbf{subword} of $Q=({j_1},\ldots,{j_N})$ is a list 
$P=(\beta_1,\ldots,\beta_N)$ such that 
$\beta_i\in \{-, j_i\}$. 

\begin{corollary}[of Theorem~\ref{thm:main}]
\label{cor:stratified}
 $\bem$ is stratified with strata given by subwords $P$ of $Q$. 
The stratum corresponding to a subword $P$ is 
$$\caS(P) =\{(x_{N+1},\ldots,x_1) \in\bem \mid x_i=x_{i+1} \text{ if } \beta_{N+1-i}=- \}.
$$
This stratum is canonically isomorphic to $\mathcal{BEM}^{Y_0,{\rm flat}(P)}$ where ${\rm flat}(P)$ is the word which deletes
all $-$ appearing in $P$.
\end{corollary}

\begin{proof} The union of these strata covers $\bem$ because 
$\caS(Q)=\bem$.
For $P=(\beta_1,\ldots,\beta_N)$ and $P'=(\beta'_1,\ldots,\beta'_N)$ define the subword 
$P\vee P'=(\gamma_1,\ldots,\gamma_N)$,
where $\gamma_i=-$ if $\beta_i$ or $\beta'_i$ equals $-$.
Then
\begin{align*}
\caS(P)\cap\caS(P') &= \{(x_{N+1},\ldots,x_1) \in\bem \mid x_i=x_{i+1} \text{ if } \beta_{N+1-i}=- \text{ or }\beta'_{N+1-i}=-\} \\
&=\caS(P\vee P').
\end{align*}
The isomorphism from $\caS(P)$ to $\mathcal{BEM}^{Y_0,{\rm flat}(P)}$
is the projection that deletes all components of $\caS(P)$ 
associated to a $-$.
\end{proof}

\section{Moment polytopes}\label{sec:mps}

The projective space $\mathbb{P}^d$ is a symplectic manifold with \emph{Fubini-Study} symplectic form.
Following \cite[Section 6.6]{CannasdaSilva2}, consider the restriction of the action of $T^d=(\C^*)^d$ on $\mathbb{P}^d$, given by coordinate-wise multiplication, to the compact real subtorus
$$
T^d_{\mathbb{R}}=\{(e^{i\theta_1},\ldots,e^{i\theta_d})\in(\mathbb{C}^*)^d\mid \theta_i\in\mathbb{R} \text{ for all }i\}.
$$
As explained in \cite[Example~4]{Knutson:symplectic} the action of $T^d_\mathbb{R}$ on $\mathbb{P}^d$ has a moment map. That is,
$\mathbb{P}^d$ is a Hamiltonian $T^d_\mathbb{R}$-manifold. 

Now let $X$ be a smooth algebraic variety with an action of a torus
$T\cong (\C^*)^n$ with $n\leq\dim(X)$ and a $T$-equivariant embedding into $\mathbb{P}^d$. 
Again, we restrict the $T$-action to the compact real subtorus $T_\mathbb{R}$.
Since $T$ is isomorphic to a subgroup of $T^d$ then \cite[p.~64; point~1.]{Knutson:symplectic} tells us that $\mathbb{P}^d$ is also a Hamiltonian $T_\mathbb{R}$-manifold.
Smoothness says $X$ is a $T$-invariant submanifold of $\mathbb{P}^d$. By \cite[p.~64; point~1.]{Knutson:symplectic}, it is a Hamiltonian $T_\mathbb{R}$-manifold.
Hence $X$ has a \emph{moment map}
\[\Phi: X \rightarrow \frt^*_\mathbb{R},\]
 where $\mathfrak{t}_\mathbb{R}^*\simeq \mathbb{R}^n$ is the dual of the Lie algebra of $T_\mathbb{R}$.  
There are only finitely many isolated fixed points since the fixed point locus $X^T$ is closed for the Zariski topology.
Therefore by \cite{Atiyah,Guillemin-Sternberg}, the image 
$\Phi(X)$ is a polytope in $\frt^*_\mathbb{R}$;
namely, it is the convex hull of the image under $\Phi$ of the $T_\mathbb{R}$-fixed points.  $\Phi(X)$ is known as the
\emph{moment polytope} of $X$. A primer on moment maps which outlines their most important properties,
including the ones we will use, can be found in \cite[Section~2.2]{Knutson:symplectic}. From now on, we will omit the subscript $\mathbb{R}$ from $T$ and the Lie algebra.

Moment map images provide a source of polytopes. 
It is natural to consider $\Phi(BS^Q)$ for $\sf T$ a maximal torus of $\sf G$, which is the moment polytope of the Bott-Samelson variety $BS^Q$. 
To our best knowledge, the first analysis of this polytope in the literature is \cite{Escobar} (an anonymous referee has suggested to us the relevance of the preprint \cite{Haerterich} who studies T-equivariant cohomology of Bott-Samelson varieties).

We are interested on the action of ${\sf S}:={\sf T}\cap {\sf K}$, a maximal  torus in ${\sf K}$, on $Y$ and on its BEM varieties.
We will show in Theorem~\ref{thm:moment} that $\Phi(\bem)$ is the convex hull of certain reflections of $\Phi_{\sf S}(BS^Q)$, where $\Phi_{\sf S}$ denotes the moment map of $BS^Q$ for the ${\sf S}$-action. The proof exploits the analogies between the descriptions of the manifolds.

In order to compute $\Phi(\bem)$ we embed $\bem$ into a product of ${\sf G}/{\sf P}_{\widehat{i}}$.
To compute $\Phi(\bem)$ it is not necessary to explicitly embed $\bem$ into projective space (via generalized Pl\"ucker embeddings followed by the Segre map). This is since
the Grassmannian ${\sf G}/{\sf P}_{\widehat{i}}$ is diffeomorphic to the coadjoint orbit $\mathcal{O}_{\omega_i}\cong {\sf G}_{\mathbb{R}}/ ({\sf G}_{\mathbb{R}})_{\omega_i}$ where $\omega_i\in\mathfrak{t}^*$ is a fundamental weight, $ {\sf G}_{\mathbb{R}}$ is a compact form of $ {\sf G}$, and $ ({\sf G}_{\mathbb{R}})_{\omega_i}$ denotes the stabilizer of $\omega_i$ under the coadjoint action, see e.g.~\cite{Caviedes}.
This coadjoint orbit is already a Hamiltonian ${\sf T}$-manifold with \emph{Kostant-Kirillov-Souriau} symplectic form and moment map 
\begin{align}
\label{eqn:729given}
\Phi_i: & \  {\sf G}_{\mathbb{R}}/ ({\sf G}_{\mathbb{R}})_{\omega_i}\to \mathfrak{g}^* \to \mathfrak{t}^*\simeq  \R^n,
\end{align}
where the first map is $ (g{\sf P}_{\widehat{i}}) \mapsto  g\omega_i$ and the second map is the projection induced from the inclusion $T\subset G$.
For each $g$ representing an element of the Weyl group of $G$ we have that $\Phi_i(g\omega_i)=g\omega_i$.

Actually, if we embed $\bem$ into projective space as indicated above we wouldn't get a different polytope anyway.
This is because the Kostant-Kirillov-Souriau form coincides with the pullback of the Fubini-Study form to ${\sf G}/{\sf P}_{\widehat{i}}$ under the ${\sf T}$-equivariant embedding given by the line bundle ${\mathcal L}(\omega_i)$,
see \cite[Remark~3.5]{Caviedes}. 

Thinking of the fundamental weights $\omega_i\in\mathfrak{t}^*$ as functions $\omega_i:\mathfrak{t}\rightarrow\mathbb{R}$,
$\omega_i|_{\mathfrak s}$ is the restriction of $\omega_i$ to $\mathfrak{s}\subset\mathfrak{t}$.

\begin{theorem}
\label{thm:moment}
$\bem$ has an embedding into a product of ${\sf G}/{\sf P}$ with $\sf P$ maximal as a symplectic submanifold of this product with Hamiltonian ${\sf S}$-action; the corresponding
moment polytope is
\begin{align*}
\Phi(\bem) &={\sf conv}\left\{x \left(\sum_{i=1}^r\omega_i|_{\mathfrak s}+\sum_{i={|Q|}}^{1}s_{j_{|Q|}}\cdots s_{j_i} \omega_i|_{\mathfrak s}\right) 
\mid \ x \in Y_0^{\sf S} \text{ and } (j_1,\ldots,j_{|Q|}) \subseteq Q \right\}\\
&={\sf conv}\left\{x \cdot  \Phi_{\sf S}(BS^Q)
\mid \ x \in Y_0^{\sf S}  \right\}.
\end{align*}
\end{theorem}
\begin{proof}
$\bem$ embeds into a product of ${\sf G}/{\sf P}$ with $\sf P$ maximal, as follows:
\begin{proposition}\label{prop:theembed} The following map is an embedding:
\begin{align*} \delta: \bem & \hookrightarrow \prod_{i=1}^r {\sf G}/{\sf P}_{\widehat{i}} \times \prod_{j=1}^{|Q|} 
{\sf G}/{\sf P}_{\widehat{i}_{|Q|-j+1}} \\
(x{\sf B},g_{|Q|}{\sf B},\ldots,g_1{\sf B}) & \mapsto(x{\sf P}_{\widehat{1}},\ldots,x{\sf P}_{\widehat{r}},g_{|Q|}{\sf P}_{\widehat{i}_{|Q|}},g_{|Q|-1}{\sf P}_{\widehat{i}_{|Q|-1}},\ldots, g_1{\sf P}_{\widehat{i_1}}).
\end{align*}
\end{proposition}

By \cite[9.2.2]{Vakil} the fibered product with a closed subscheme is a closed subscheme of the fibered product.
From this it immediately follows that BEM is a closed subscheme of a fibered product of ${\sf G}/{\sf B}$'s.
It is well known that this fibered product embeds into the desired product of ${\sf G}/{\sf P}_{\widehat{i}}$.
Thus the proposition follows.
To be more explicit, we include the proof below.

\begin{proof}
First to see that $\delta$ is injective suppose
$\delta(x{\sf B},a_{|Q|}{\sf B},\hdots,a_1{\sf B}) = \delta(y{\sf B},b_{|Q|}{\sf B},\hdots,b_1{\sf B})$, 
then 
$x{\sf P}_{\widehat{i}} = y{\sf P}_{\widehat{i}}$ for $1\leq i\leq r$ and therefore 
$y^{-1}x \in \bigcap_{i=1}^r {\sf P}_{\widehat{i}} = {\sf B}$.
Thus, $x{\sf B} = y{\sf B}$.  
Next, the assumption
\begin{equation}
\label{eqn:729abc}
a_{|Q|}{\sf P}_{\widehat{i_{|Q|}}} =
b_{|Q|}{\sf P}_{\widehat{i_{|Q|}}}
\implies b_{|Q|}^{-1}a_{|Q|}\in {\sf P}_{\widehat{i_{|Q|}}}.
\end{equation}
Also, using the definition of $\bem$ (\ref{eqn:Magyarhelper}),
\begin{equation}
\label{eqn:729xyz}
a_{|Q|}{\sf P}_{i_{|Q|}} = x{\sf P}_{i_{|Q|}} = y{\sf P}_{i_{|Q|}} = b_{|Q|}{\sf P}_{i_{|Q|}}\implies
b_{|Q|}^{-1}a_{|Q|} \in {\sf P}_{i_{|Q|}}.
\end{equation}
Combining (\ref{eqn:729abc}) and (\ref{eqn:729xyz}) gives
\[b_{|Q|}^{-1}a_{|Q|}\in
{\sf P}_{\widehat{i_{|Q|}}}\cap {\sf P}_{i_{|Q|}}={\sf B}
\implies
a_{|Q|}{\sf B} = b_{|Q|}{\sf B}.\]  
Reasoning similarly, we see that 
$a_k{\sf B}=b_k{\sf B}$ for all $k=|Q|-1, |Q|-2,\hdots,1$, as required. Thus $\delta$ is injective.

It is well known that the map
\[{\sf G}/{\sf B}\rightarrow\prod_{i=1}^r{\sf G}/{\sf P}_{\widehat{i}}: x{\sf B}\mapsto
(x{\sf P}_{\widehat{1}},\ldots,x{\sf P}_{\widehat{r}})\]
is an embedding of algebraic varieties. Consequently, the map
\begin{align*}
\kappa:({\sf G}/{\sf B})^{|Q|+1}&\hookrightarrow \prod_{m=1}^{|Q|+1}\prod_{i=1}^r{\sf G}/{\sf P}_{\widehat{i}}\\
(x_{|Q|+1}{\sf B},x_{|Q|}{\sf B},\ldots,x_1{\sf B}) & \mapsto
((x_{|Q|+1}{\sf P}_{\widehat{1}},\ldots,x_{|Q|+1}{\sf P}_{\widehat{r}}),
\ldots,(x_{1}{\sf P}_{\widehat{1}},\ldots,x_{1}{\sf P}_{\widehat{r}}))
\end{align*}
is also an embedding.
Let $Q=(q_1,q_2,\ldots,q_N)$.
The image of $\bem\subset ({\sf G}/{\sf B})^{|Q|+1}$ under $\kappa$ satisfies  
\begin{equation}\label{eq:bemRepeat}
x_m{\sf P}_{\widehat{i}}=x_{m+1}{\sf P}_{\widehat{i}}
\end{equation}
whenever $i\neq q_m$, for $m=1,2,\ldots,|Q|$. Thus $\delta$ factors:
\begin{center}
\begin{tikzpicture}
\node (top) at (0,0) {$\bem$};
\node (delta) at (2,-.75) {$\delta$};
\node (kappa) at (1.7,.2) {$\kappa$};
\node (pi) at (4.2,-.9) {$\psi$};
\node (grgr) at (4,0) {$\prod_{m=1}^{|Q+1|}\prod_{i=1}^r{\sf G}/{\sf P}_{\widehat{i}}$};
\node (gr) at (4,-2) {$\prod_{i=1}^r {\sf G}/{\sf P}_{\widehat{i}} \times \prod_{j=1}^{|Q|} 
{\sf G}/{\sf P}_{\widehat{i_{|Q|-j+1}}}$}; 
  \draw [->] (top) -- (grgr);
    \draw[->] (grgr)--(gr);
        \draw[->] (top)--(gr);
\end{tikzpicture}
\end{center}
where $\psi$ is the projection that forgets the repetitions of \eqref{eq:bemRepeat}.
Thus, $\delta$ is an embedding.
\end{proof}

$Gr:=\prod_{i=1}^r {\sf G}/{\sf P}_{\widehat{i}} \times \prod_{j=1}^{|Q|} 
{\sf G}/{\sf P}_{\widehat{i_{|Q|-j+1}}}$
is naturally a symplectic manifold, and is Hamiltonian with respect to the
(diagonal) action of ${\sf T}$.
By \cite[p.~64; point 1.]{Knutson:symplectic} the same is true for this action restricted to the subtorus ${\sf S}$.
As a submanifold of $Gr$, $\bem$ is also symplectic, and is clearly
stable under the ${\sf S}$-action.  From this it follows (cf. \cite[p.~65; point 4.]{Knutson:symplectic})
that the ${\sf S}$-action is Hamiltonian, whence $\bem$ has a moment map $\Phi$.  Then one sees from
\cite[p.~64; point 1. and p.~65; point 3.]{Knutson:symplectic} that $\Phi$ is given by
\begin{equation}
\label{eqn:729thesum}
 \bem \hookrightarrow Gr   \lra{\sum \Phi_i}\mathfrak{t}^*\lra{} \mathfrak{s}^*, 
\end{equation}
where $\Phi_i: {\sf G}/{\sf P}_{\widehat{i}} \rightarrow \mathfrak{t}^*$ is the moment map for ${\sf G}/{\sf P}_{\widehat{i}}$ and 
$\mathfrak{t}^*\rightarrow \mathfrak{s}^*$ is induced from the inclusion ${\sf S}\subset {\sf T}$.
The second map restricts functions $\mathfrak{t}\rightarrow \mathbb{R}$ to $\mathfrak{s}$.
Therefore, by (\ref{eqn:729given}) and (\ref{eqn:729thesum}) combined,
 the moment map $\Phi: \bem \rightarrow \mathfrak{s}^*$ is given by 
\begin{equation}\label{eqn:729combinedmoment}
 (x{\sf B},g_{|Q|}{\sf B},\ldots,g_1{\sf B})  \lmp{} 
\sum_{i=1}^r x\omega_i|_{\mathfrak s}+\sum_{i=1}^{|Q|}g_i\omega_i|_{\mathfrak s}. 
\end{equation}

\begin{proposition}[${\sf S}$-fixed points of $\bem$]
\label{prop:thefixed}
The ${\sf S}$-fixed points of $\bem$
are indexed by pairs $(x{\sf B},J)$, where $x{\sf B}$ is a ${\sf T}$-fixed point of ${\sf G}/{\sf B}$ contained
in $Y_0$, and $J=(\beta_1,\ldots,\beta_{|Q|})$ is a subword of $Q$.
Indeed, the fixed points are precisely the points
\begin{equation}
p_{(x{\sf B},J)}:=(x{\sf B},xs_{\beta_{|Q|}}{\sf B},xs_{\beta_{|Q|}}s_{\beta_{|Q|-1}}{\sf B},\ldots,xs_{\beta_{|Q|}}\cdots s_{\gb_1}{\sf B})\in \bem,
\end{equation}
where $s_{\gb_i}$ is the identity if $\gb_i=-$.
\end{proposition}
\begin{proof}
We first verify that 
$p_{(x{\sf B},J)}\in \bem$.
Note that for $i=1,\ldots,|Q|$, since by (\ref{eqn:standardpara}),
${\sf B}s_{\beta_i}{\sf B}\in{\sf P}_{j_i}$, in particular $s_{\beta_i}\in {\sf P}_{j_i}$ and hence
$$xs_{\beta_{|Q|}}\cdots s_{\beta_{i}}{\sf P}_{\beta_{i}}=xs_{\beta_{|Q|}}\cdots s_{\beta_{i-1}}{\sf P}_{\beta_{i}}.$$ 
Therefore $(xs_{\beta_{|Q|}}\cdots s_{\beta_{i-1}},xs_{\beta_{|Q|}}\cdots s_{\beta_i})$ satisfies \eqref{eqn:Magyarhelper} for $i=1,\ldots,|Q|$,
as needed.

Since 
\begin{equation}
\label{eqn:SisTfixed}
({\sf G}/{\sf B})^{\sf S}=({\sf G}/{\sf B})^{\sf T}
\end{equation}
(see \cite[pg.~128]{Brion1999}), the ${\sf S}$-fixed points of $Y_0$
are of the form $x{\sf B}$ where $x\in N_{\sf G}({\sf T})$. Therefore, 
\begin{equation}
\label{eqn:itnormals}
xs_{\beta_{|Q|}}\cdots s_{\beta_{i}}\in N_{\sf G}({\sf T}).
\end{equation}
Moreover,  
since ${\sf S}\subset {\sf T}$, for ${\sf t}\in{\sf S}$ we have 
by (\ref{eqn:itnormals}) that
$$
{\sf t}\cdot xs_{\beta_{|Q|}}\cdots s_{\beta_{i}}{\sf B}=xs_{\beta_{|Q|}}\cdots s_{\beta_{i}}{\sf B},$$
so $p_{(x{\sf B},J)}$ is an ${\sf S}$-fixed point.

Conversely, suppose $(x_{|Q|+1}{\sf B},x_{|Q|}{\sf B},\ldots,x_{1}{\sf B})$ is an ${\sf S}$-fixed point of $\bem$. Clearly, 
$x_{|Q|+1}{\sf B}\in (Y_0)^{\sf S}$.
By (\ref{eqn:SisTfixed}), each $x_i{\sf B}$ is a ${\sf T}$-fixed point so we may assume 
\begin{equation}
\label{eqn:wemayassume729}
x_i\in N_{\sf G}({\sf T}).
\end{equation}
By the definition (\ref{eqn:Magyarhelper}) of $\bem$,
$x_i{\sf P}_{{j_i}}=x_{i-1}{\sf P}_{{j_i}}$. 
Thus,
 $x_{i-1}^{-1}x_i\in{\sf P}_{{j_i}}$. 
Hence, in view of (\ref{eqn:wemayassume729}) we may further assume that
$x_{i-1}^{-1}x_i\in\{id, s_{\ga_{j_i}}\}$.
Therefore $(x_{|Q|+1}{\sf B},x_{|Q|}{\sf B},\ldots,x_{1}{\sf B})$ is of the form $p_{
(x{\sf B},J)}$, as asserted.
\end{proof}

Since $\Phi(\bem)$ is the convex hull of $\Phi$ applied to this set of points, the first equality of the theorem holds by 
Proposition~\ref{prop:thefixed} combined with
(\ref{eqn:729combinedmoment}).

Similar arguments \cite{Escobar} show
the moment polytope of a Bott-Samelson variety is 
\begin{equation}\label{eq:bsmoment}
\Phi(BS^Q) ={\sf conv}\left\{\sum_{i=1}^r\omega_i+\sum_{i={|Q|}}^{1}s_{\beta_{|Q|}}\cdots s_{\beta_i} \omega_i
\mid \ (\beta_1,\ldots,\beta_{|Q|}) \text{ is a subword of } Q \right\}.
\end{equation}
The second equality follows by restricting the weights to $\mathfrak{s}$.
\end{proof}

Define the {\bf BEM polytope} ${\mathcal P}_{Y_0,Q}$ as $\Phi(\mathcal{BEM}^{Y_0,Q})$.
In Section~\ref{sec:furtherPQ} we study ${\mathcal P}_{Y_0,Q}$ for the symmetric pair ${\sf G}={\sf GL}_{p+q}$ and ${\sf K}={\sf GL}_{p}\times{\sf GL}_{q}$.

We remark it would be interesting to study the polytopes coming from the ${\sf K}$-action on $\bem$.
$\bem$ is a Hamiltonian ${\sf K}$-manifold and therefore has a moment map
$\Phi_{\sf K}$.
\cite[Section 2.5]{Knutson:symplectic} describes two polytopes which are
associated with the image of $\Phi_{\sf K}$.
One of these is the intersection of the image of $\Phi_{\sf K}$
with the positive Weyl chamber. Kirwan's noncommutative convexity Theorem
\cite{Kirwan} states that this intersection is a polytope.

\section{The Barbasch-Evens-Magyar varieties in type $A$}\label{sec:A-BEMs}

In this section we describe the $\sf K$-orbits and Barbasch-Evens-Magyar varieties for the case of symmetric pairs $({\sf G},{\sf K})$ where ${\sf G}$ is a general linear group.
The three possible pairs are $({\sf GL}_{p+q}, {\sf GL}_p\times {\sf GL}_q)$, $({\sf GL}_{2n},{\sf Sp}_{2n})$, and $({\sf GL}_n,{\sf O}_n)$.
All of these symmetric subgroups are connected, see e.g. \cite{Wyser-13}.

When ${\sf G}={\sf GL}_n$ we take the simple roots to be 
$\Delta=\{\alpha_i=\vec e_i-\vec e_{i+1}|1\leq i\leq n-1\}$,
where $\vec e_i\in{\mathbb R}^n$ is the standard basis vector.
With this choice of root system embedding, we may identify the
fundamental weight $\omega_i$ with the vector $\sum_{k=1}^i \vec e_i$.
${\sf W}={\mathfrak S}_n$ is identified with the symmetric group of permutations on $\{1,2,\ldots,n\}$.
Thus, $s_i$ is the simple transposition interchanging $i$ and $i+1$.

\subsection{$\sf K$-orbits of type A}\label{sec_typeA_BEM}

The running example of this section is $({\sf G},{\sf K})=({\sf GL}_{p+q},{\sf GL}_p \times {\sf GL}_q)$.
Let $n=p+q$, and consider the involution $\theta$ of ${\sf G}={\sf GL}_n$ defined by conjugation using the diagonal
matrix having $p$-many $1$'s followed by $q$-many $-1$'s.
Then 
${\sf K} = {\sf G}^{\theta} \cong {\sf GL}_p \times {\sf GL}_q$, 
embedded as block diagonal matrices with an upper-left
invertible $p \times p$ block, a lower-right invertible $q \times q$ block, and zeros outside of these blocks.

The orbits in this case are parametrized by $(p,q)$-clans \cite{Matsuki, Yamamoto}, which we now describe. 
A $(p,q)$-\textbf{indication} is a string of characters $\gamma=c_1 \hdots c_n$ with each $c_i\in\{+,-\}\cup{\mathbb Z}_{>0}$ and such that
\begin{itemize}
\item
if $c_i\in \Z_{>0}$, then there is a unique $j\neq i$ such that $c_i=c_j$; and 
\item 
$ \#\{i\mid c_i=+\}-\#\{i\mid c_i=-\}=p-q$.
\end{itemize}
Now, consider the equivalence relation between indications given by $\gamma\sim\gamma'$ if and only if 
\begin{itemize}
\item $c_i=c'_i$ whenever $c_i\in\{+,-\}$; and 
\item
there exists a bijection $f:\Z_{> 0}\to \Z_{> 0}$ such that $c_i=f(c'_i)$ for all $i$ with $c_i\in\Z_{> 0}$.
\end{itemize}
The $(p,q)$-\textbf{clans} are the equivalence classes of this equivalence relation.
By slight abuse of notation, we use the same notation for indications to denote clans; for example $1+-1$ is the same clan as $2+-2$.
Let ${\tt Clans}_{p,q}$ be the set of $(p,q)$-clans. 
The closed orbits are indexed by {\bf matchless clans}, i.e., clans using only $+,-$. Lemma~\ref{lemma:closedisflag} implies these closed orbits are isomorphic to $({\sf GL}_p/{\sf B})\times ({\sf GL}_q/{\sf B})$, the product of two flag varieties.

We briefly remark that each $(p,q)$-clan corresponds to an involution in ${\mathfrak S_{p+q}}$.
Indeed, given a clan $\gamma$ we obtain an involution $w$ by letting $w_i=i$ whenever $c_i\in\{+,-\}$ and $w_i=j$ whenever $c_i=c_j\in\Z_{>0}$.

Next, we explicitly describe the orbit closures $Y_{\gamma}$.
Fix  
$\gamma=c_1 \hdots c_n\in {\tt Clans}_{p,q}$.
For $i=1,\hdots,n$, define:
\begin{itemize}
	\item $\gamma(i;+) = \#\{ (j,k) \mid c_j=c_k\in\Z_{>0},\ 1\le j<k\le i\}+\#\{j\mid c_j=+,\ 1\leq j\leq i\}$; and
	\item $\gamma(i;-) = \#\{ (j,k) \mid c_j=c_k\in\Z_{>0},\ 1\le j<k\le i\}+\#\{j\mid c_j=-,\ 1\leq j\leq i\}$.
\end{itemize}
For $1 \leq i < j \leq n$, define
\begin{itemize}
	\item $\gamma(i;j) = \#\{k \in [1,i] \mid c_k = c_{\ell} \in \N \text{ with } \ell > j\}$.
\end{itemize}
Let $E_p={\rm span}\{\vec e_1,\vec e_2,\ldots,\vec e_p\}$
be the span of the first $p$ standard basis vectors, and let
$E^q={\rm span}\{\vec e_{p+1}, \vec e_{p+2},\ldots, \vec e_n\}$
be the span of the last $q$ standard basis vectors.  Let
$\rho: \C^n \rightarrow E_p$
 be the projection map onto the subspace $E_p$.

\begin{figure}[h]
\begin{center}
\includegraphics[scale=.3]{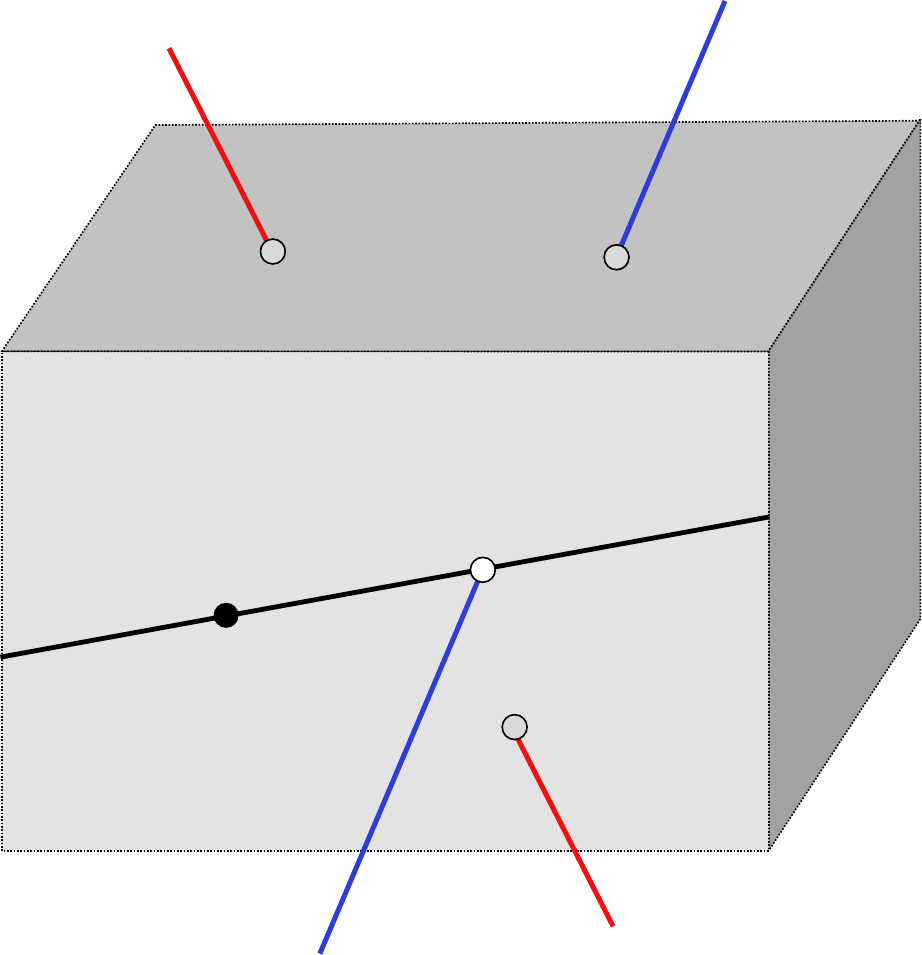} 
\caption{$Y_{1+-1}$}\label{fig:1+-1}
\put(-35,140){$E_2$}
\put(-116,134){$E^2$}
\put(-110,55){$V_1$}
\put(-134,75){$V_3$}
\put(-50,65){$V_2$}
\end{center}
\end{figure}

Suppose $\gamma \in {\tt Clans}_{p,q}$ and $\e\in {\tt Clans}_{r,s}$.  Then 
$\e=\e_1 \ldots \e_{r+s}$
{\bf (pattern) avoids} $\gamma=\gamma_1 \ldots \gamma_{p+q}$ if there are
no indices $i_1<i_2<\cdots<i_{p+q}$ such that:
\begin{enumerate}
 \item if $\gamma_j = \pm$ then $\e_{i_j}=\gamma_j$; and
 \item if $\gamma_k=\gamma_{\ell}$ then $\e_{i_k}=\e_{i_{\ell}}$.
\end{enumerate}
A clan $\gamma$ is {\bf noncrossing} if $\gamma$ avoids
$1212$.

\begin{theorem}[{\cite[Corollary 1.3]{Wyser-16}}, {\cite[Remark~3.9]{Richthing}}]
\label{thm:wyserbruhat}  
The $\sf K$-orbit closure
$Y_{\gamma}$ is the set of flags $F_{\bullet}=(V_1,\ldots,V_n)$ such that:
\begin{enumerate}
	\item $\dim(V_i \cap E_p) \geq \gamma(i;+)$ for all $i$;
	\item $\dim(V_i \cap E^q) \geq \gamma(i;-)$ for all $i$.
	\item $\dim(\rho(V_i) + V_j) \leq j + \gamma(i;j)$ for all $i<j$.
\end{enumerate}
If $\gamma$ is noncrossing, the third condition is redundant.
\end{theorem}

\begin{example}
Let $p=q=2$ and $\gamma=1+-1$ (a noncrossing clan). 
In fact, $Y_{\gamma}=s_1\star s_3\star s_2 \star Y_{++--}$
 and
\[Y_\gamma=\{(V_1,V_2,V_3)\in {\sf Gr}(1,4)\times {\sf Gr}(2,4)\times {\sf Gr}(3,4) \mid \dim(V_2\cap E_2)\geq 1, \dim(V_3\cap E^2)\geq 1\}.\]
A projectivized depiction of a general point in this orbit closure is given in Figure~\ref{fig:1+-1}. The blue and red lines represent $E_2$ and $E^2$ respectively. The moving flag $(V_1,V_2,V_3)$ is the (black point, black line, front face).\qed
\end{example}

\begin{figure}[h]
\begin{center}
\includegraphics[scale=.3]{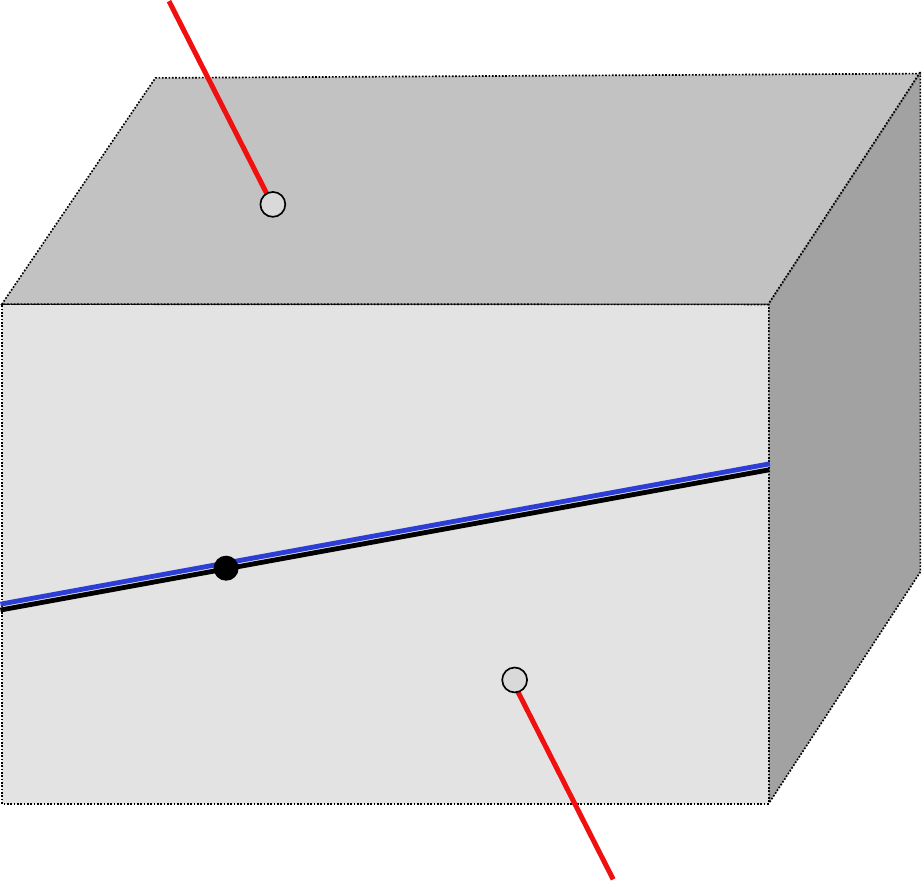} 
\caption{$Y_{++--}$}\label{fig:++--}
\put(-48,46){$E_2$}
\put(-116,130){$E^2$}
\put(-110,51){$V_1$}
\put(-135,73){$V_3$}
\put(-48,63){$V_2$}
\end{center}
\end{figure}

W.~McGovern characterized the singular orbit closures:

\begin{theorem}[\cite{McGovern}]
\label{thm:McGovern}
$Y_{\gamma}$ is smooth if and only if $\gamma$ avoids
the patterns
$1+-1, \ 1-+1, \ 1212, \ 1+221, \ 1-221, \ 122+1$, $122-1, \ 122331$.
\end{theorem}

\begin{example}
By Theorem~\ref{thm:McGovern}, $Y_{1+-1}$ is singular. 
One computes (e.g, using the methods
of \cite{WWY}) that the singular locus is the closed orbit $Y_{++--}$ where $V_2=E_2$ (the black and blue lines agree). In Figure \ref{fig:1+-1}, the general picture of $Y_{1+-1}$, the black line $V_2$ has three degrees of freedom to move. Now consider the picture of $Y_{++--}$ (Figure \ref{fig:++--}). Pick any point of the blue line $E_2$.
Then the black line $V_2$ has two degrees of freedom to pivot and
remain inside $Y_{1+-1}$. This is true of any other point as well. Informally, this additional degree of freedom is singular behavior.\qed
\end{example}

\subsection{Barbasch-Evens-Magyar varieties of type A} 
In Section~\ref{sec:BEManyG} we gave our general definition of the
Barbasch-Evens-Magyar varieties. We now describe, using diagrams, these configuration spaces for the special
case of symmetric pairs $({\sf G},{\sf K})$ where ${\sf G}$ is a general linear group. 

\begin{wrapfigure}{l}{0.2\textwidth}
\begin{center}
  \begin{tikzpicture}
\node (top) at (0,0) {$\C^3$};
    \node [below of=top] (flag2)  {$F_2$};
    \node [below of=flag2] (flag1) {$F_1$};
    \node [right of=flag1] (L1) {$V_4$};
        \node [right of=flag2] (P1) {$V_3$};
        \node [right of=L1] (L2) {$V_2$};
\node [right of=P1] (P2) {$V_1$};
\node [below of=flag1] (down) {$\C^0$};
\draw [thick,shorten <=-2pt] (top) -- (flag2);
\draw [thick,shorten <=-2pt] (flag2) -- (flag1);
\draw [thick,shorten <=-2pt] (P1) -- (L1);
\draw [thick,shorten <=-2pt] (P1) -- (L2);
\draw [thick,shorten <=-2pt] (P2) -- (L2);
\draw [thick,shorten <=-2pt] (top) -- (P1);
\draw [thick,shorten <=-2pt] (top) -- (P2);
\draw [thick,shorten <=-2pt] (flag2) -- (L1);
\draw [thick,shorten <=-2pt] (flag1) -- (down);
\draw [thick,shorten <=-2pt] (L1) -- (down);
\draw [thick,shorten <=-2pt] (L2) -- (down);
\end{tikzpicture}
\end{center}
\end{wrapfigure} 

Let $Q=(j_1,j_2,\ldots,j_N)$ be a word and $\sf F\subseteq {\sf G}/{\sf B}$. We first define the configuration space $\mathcal{C}(Q,{\sf F})$.
Informally, a point of $\mathcal{C}(Q,{\sf F})$ is
a collection of vector spaces forming a diagram.
On the left is an example for ${\sf G}={\sf GL}_3$, ${\sf F}={\sf G}/{\sf B}$, and $Q=(2,1,2,1)$. The edges indicate
containments among the vector spaces.  For instance, we have $\C^0 \subset
F_1\subset F_2\subset \C^3$, as well as 
${\mathbb C}^0\subset V_2\subset V_1\subset \C^3$, etc. We require that
the flag $\C^0 \subset
F_1\subset F_2\subset \C^3$ be in ${\sf F}$.

To be precise, to define $\mathcal{C}(Q,{\sf F})$ we start
with a vertical chain whose $n+1$ vertices are labelled by 
the vector spaces $\C^0, F_1, F_2,\ldots F_{n-1}, \C^n$, from
south to north, such that the corresponding flag is an element of $\sf F$. The {\bf dimension} of a vertex is the 
dimension of the labelling vector space.
At the start, this chain is declared to be the {\bf right border
of the diagram}. 

We now grow the diagram as follows. Consider the last letter $j_N$ of $Q$.
Introduce a new vertex on the right of the diagram, labelled by $V_{N}$ of dimension
$j_N$ with edges between the vertices of dimension $j_N-1$
and $j_N+1$ (thus indicating the containment relation 
$F_{j_N-1}\subset V_N\subset F_{j_N+1}$). 
We modify the current right border by replacing the vertex of the current right border of dimension $j_N$ with the new vertex labelled by $V_N$.
Now repeat successively with $j_{N-1},j_{N-2},\ldots j_2,j_1$. At step $k$, a new vertex
labelled by $V_{N-k+1}$ is added to the right of the right border, of dimension $j_{N-k+1}$, and becomes the new member of the right border, replacing the unique
vertex of dimension $j_{N-k+1}$ of the current right border.
Note that the letter $j_{N-k+1}$ of $Q$ corresponds to the \textbf{quadrangle} in Figure~\ref{fig_quad}
	\begin{figure}
	\begin{tikzpicture}
	\node (top) at (0,0) {$V_a$};
    \node [below of=top] (flag3)  {$V_b$};
    \node [below of=flag3] (flag2) {$V_c$};
        \node [right of=flag3] (P1) {\qquad$V_{N-k+1}$};
\draw [thick,shorten <=-2pt] (top) -- (flag3);
\draw [thick,shorten <=-2pt] (flag3) -- (flag2);
\draw [thick,shorten <=-2pt] (P1) -- (top);
\draw [thick,shorten <=-2pt] (P1) -- (flag2);
	\end{tikzpicture}
	\caption{Quadrangle corresponding to  the letter $j_{N-k+1}$ of $Q$.}
	\label{fig_quad}
	\end{figure}
where $V_a$, $V_b$, and $V_c$ lie on the right border before $V_{N-k+1}$ is added, and $\dim(V_b)=j_{N-k+1}$.

Finally, a point in  $\mathcal{C}(Q,{\sf F})$ is a collection of vector spaces arranged in the diagram described, where the dimension of a vector space equals the length of a shortest path to $\mathbb{C}^0$ and $V_i\subset V_j$ whenever there is an upward edge $V_i\-- V_j$.
For the example above,
	\begin{align*}
	\mathcal{C}(Q,{\sf F})
	=
	\{(F_1,F_2,V_1,\ldots,V_4)\mid F_1,&V_4\subset F_2,\  V_4,V_2\subset V_3,\ V_2\subset V_1 \}
	\\
	&\subset {\sf Gr}(1,3)\times {\sf Gr}(2,3)\times {\sf Gr}(2,3)\times {\sf Gr}(1,3)\times {\sf Gr}(2,3)\times {\sf Gr}(1,3).
	\end{align*}

The above diagram extends the configuration space used in \cite{Magyar} to construct the Bott-Samelson variety. The difference is that \cite{Magyar} takes the initial chain to be a $\sf B$ or ${\sf GL}_n$-orbit, while here we take any subset $\sf F$. For Bott-Samelson varieties, the initial chain corresponds to a point (usually the standard basis flag).

The following result interprets the ${\sf G}={\sf GL}_n$ case of Theorem~\ref{thm:main}(I):

\begin{theorem} 
\label{thm:A}
The configuration space $\mathcal{C}(Q,Y_0)$ is the image of $\bem$ under the map $\delta$ of Proposition~\ref{prop:theembed}. 
Therefore, $\mathcal{C}(Q,Y_0)$ is isomorphic as a ${\sf K}$-variety to $ {\mathcal B}{\mathcal E}^{Y_0,Q}$.
\end{theorem}

\begin{proof}
In type $A$, the map $\delta$ may be interpreted as listing the vector spaces on the flags of successive right borders of the 
diagram for $\bem$, but avoiding
redundancy by listing only the additional new vector space introduced at each step. 
Thus the first part of the Theorem follows.
The isomorphism between $ {\mathcal B}{\mathcal E}^{Y_0,Q}$ and $\mathcal{C}(Q,Y_0)$ is the composition of the map in Theorem~\ref{thm:main} and with $\delta$. 
\end{proof}

\begin{definition}[Barbasch-Evens-Magyar variety for the symmetric pairs $({\sf GL}_n,{\sf K})$]
\label{def:BEM}
Suppose that 
$Q=({j_1},{j_2},\ldots,{j_N})$
is a (not necessarily reduced) word 
and $Y_0$ is a closed $\sf K$-orbit. Then we abuse notation and denote $\mathcal{C}(Q,Y_0)$ by $\bem$.
\end{definition}

Consider the map from $\bem$ to ${\sf G}/{\sf B}$ that maps a point in the configuration space to
the rightmost flag (corresponding to the rightmost border) in the diagram.  For example, the point of $\bem$ depicted by the example diagram above maps to the
flag $\C^0\subset V_2\subset V_1\subset \C^3$. 
The image of this map is a $\sf K$-orbit closure. Moreover, every $\sf K$-orbit closure is the image of such a map for some $\bem$. In fact, this map agrees with the map
$\varphi$ defined in (\ref{eqn:theta987}).

To complete the description of $\bem$ for the three type $A$ cases we require a description of the flags in the closed orbit $Y_0$, i.e., 
which flags may occur on the left hand side of the diagram.

In the case $({\sf G},{\sf K})=({\sf GL}_{p+q}, {\sf GL}_p\times {\sf GL}_q)$ the closed orbits are indexed by
matchless clans, i.e., $\gamma$ consists of
 $p$ $+$'s and $q$ $-$'s. The description of these orbits is given by Theorem~\ref{thm:wyserbruhat}. 
Since matchless clans are clearly noncrossing, the third condition is redundant.
In the case 
$({\sf G},{\sf K})=({\sf GL}_{2n},{\sf Sp}_{2n})$, there is a unique closed orbit $Y_0$ 
\cite[Proposition~2.4.1]{Wyser_12}.  This closed orbit is 
isomorphic to the flag variety for ${\sf K}={\sf Sp}_{2n}$ by Lemma~\ref{lemma:closedisflag}.
In the case of 
$({\sf G},{\sf K})=({\sf GL}_n,{\sf O}_n)$ with $n$ there is a unique closed orbit $Y_0$ \cite[Proposition~2.2.1 and Remark~2.3.3]{Wyser_12}.
Again, this orbit is isomorphic to the flag variety for ${\sf O}_n$.
For sake of brevity,
we refer the reader to \cite[Section~2]{Wyser-13} and the references therein for a linear algebraic description of the points
of the closed orbits in these cases.

\begin{example}
The diagram for $\bem$ where $Y_0=Y_{++--}$ and $Q=(1,3,2)$ is
\begin{figure}[h]\centering
  \begin{tikzpicture}
\node (top) at (0,0) {$\C^4$};
    \node [below of=top] (flag3)  {$F_3$};
    \node [below of=flag3] (flag2) {\hspace{-1cm}$E_2=F_2$};
    \node [below of=flag2] (flag1) {$F_1$};
    \node [right of=flag2] (L1) {$V_3$};
        \node [right of=flag3] (P1) {$V_2$};
                \node [right of=flag1] (V1) {$V_1$};
\node [below of=flag1] (down) {$\C^0$};
\draw [thick,shorten <=-2pt] (top) -- (flag3);
\draw [thick,shorten <=-2pt] (flag3) -- (flag2);
\draw [thick,shorten <=-2pt] (flag2) -- (flag1);
\draw [thick,shorten <=-2pt] (P1) -- (L1);
\draw [thick,shorten <=-2pt] (top) -- (P1);
\draw [thick,shorten <=-2pt] (flag3) -- (L1);
\draw [thick,shorten <=-2pt] (flag1) -- (down);
\draw [thick,shorten <=-2pt] (L1) -- (flag1);
\draw [thick,shorten <=-2pt] (down) -- (V1);
\draw [thick,shorten <=-2pt] (L1) -- (V1);
\end{tikzpicture}
\end{figure}

\noindent The depiction of this variety is given in Figure~\ref{fig:3}.
Here $V_1,V_3,V_2$ are given by the (projectivized) green point, line and plane respectively. The green spaces have the same incidence relations as 
the moving (black) flag in $Y_{1+-1}$. Thus, the projection forgetting all except the green spaces maps to
$Y_{1+-1}$.\qed

\begin{figure}[h]\centering
\includegraphics[scale=.3]{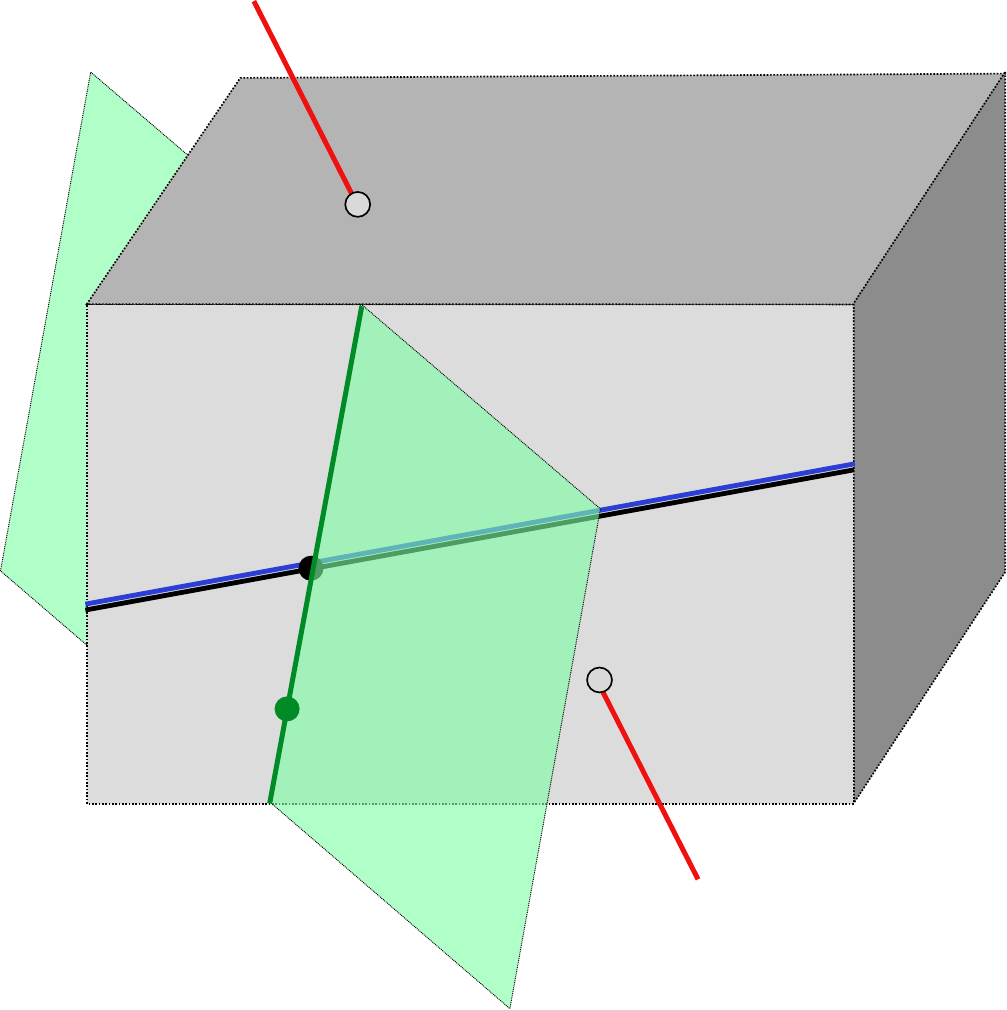} 
\caption{\label{fig:3} $\mathcal{BEM}^{Y_{++--},(1,3,2)}$: the green flag maps to $Y_{1+-1}$}
\put(-48,64){$E_2$}
\put(-116,147){$E^2$}
\put(-115,68){$F_1$}
\put(-134,90){$F_3$}
\put(-48,80){$F_2$}
\put(-120,44){$V_1$}
\put(-106,10){$V_2$}
\put(-111,85){$V_3$}
\end{figure}

\end{example}

\begin{example}
Let us compute the Poincar\'e polynomial $p_{Y_0,Q}(z)$ of $\bem$ when ${\sf K}={\sf GL}_p\times {\sf GL}_{q}$.
Recall that Corollary~\ref{cor: poincare} says that $p_{Y_0,Q}(z)=r_{Y_0}(z)(1+z)^N$, where $N$ is the length of $Q$.
Let 
$[n]_z!=[1]_z [2]_z\cdots [n]_z$
where $[i]_z=1+z+z^2+\cdots +z^{i-1}$.
The Poincar\'e polynomial of ${\sf GL}_n/{\sf B}$ is $[n]_z!$.
We now have $r_{Y_0}(z)=[p]_z![q]_z!$ for any choice of closed orbit $Y_0$.
This is since by Lemma~\ref{lemma:closedisflag}, we have
$Y_0\cong ({\sf GL}_p/{\sf B}) \times ({\sf GL}_q/{\sf B})$, combined with 
the K\"{u}nneth formula. 
\end{example}

The standard maximal torus ${\sf T}\cong (\mathbb{C}^*)^n$ in ${\sf GL}_n$ consists of invertible diagonal matrices.
There is a natural ${\sf K}$-action on $\bem$, described in Section~\ref{sec:BEManyG}, which induces an ${\sf S}$-action, where ${\sf S}={\sf T}\cap {\sf K}$.
Let us describe this action in the present setting. 
A matrix in ${\sf K}$ acts on the Grassmannian of $m$-dimensional subspaces of $\mathbb{C}^n$ by change of basis.
We extend this to an action of ${\sf K}$ on $\bem$ diagonally:
\[{\sf k}\cdot (F_1,F_2,\ldots,F_{n-1},V_1,V_2,\ldots V_N)
=({\sf k}\cdot F_1, {\sf k}\cdot F_2,\ldots {\sf k}\cdot F_{n-1},
{\sf k}\cdot V_1, {\sf k}\cdot V_2,\ldots, {\sf k}\cdot V_N),\]
where ${\sf k}\in{\sf K}$.

In Section~\ref{sec:furtherPQ} we study the moment polytope of $\bem$ for the pair $({\sf GL}_{p+q},{\sf GL}_p\times {\sf GL}_q)$.
To do so, we utilize the following concrete description of the ${\sf S}$-fixed points of $\bem$ for this symmetric pair.
Notice that in this case ${\sf S}={\sf T}$.
A \textbf{subword} of $Q=(j_1,\ldots,j_N)$ is a list 
$P=(\beta_1,\ldots,\beta_N)$ such that 
$\beta_i\in \{-, j_i\}$.
As we saw above, each letter of $Q$ corresponds to a quadrangle of the associated diagram. 
A subword $P$ corresponds to a subset of these quadrangles. 
Concretely, for each quadrangle in the set, require
the two vertices associated to vector spaces of equal dimension to be the same space.  For each quadrangle not in the set, insist those same vector spaces be different. Call such an assignment
given a left border associated to a flag $F_{\bullet}$ a \textbf{$\boldsymbol{P}$-growth of $\boldsymbol{F_{\bullet}}$}.

Given a matchless clan $\gamma$, a permutation $\sigma\in {\mathfrak S}_{p+q}$ is $\gamma$-{\bf shuffled} if it assigns
\begin{itemize}
\item $1,2,\ldots,p$ in any order to the $+$'s;
\item $p+1,p+2,\ldots,n$ in any order to the $-$'s.
\end{itemize}
Hence there are $p!q!$ such permutations (independent of $\gamma$).

Associated to any $\gamma$-shuffled permutation define $F_{\bullet}^{\gamma,\sigma}$
to be the $\sigma$-permuted coordinate flag, i.e., the one whose $d$-dimensional
subspace is $\langle \vec e_{\sigma(1)},\ldots\vec e_{\sigma(d)}\rangle$.

We will use this result, due to A.~Yamamoto:
\begin{proposition}[\cite{Yamamoto}]\label{prop:729yama}
The ${\sf S}$-fixed points of $Y_{\gamma}$ are flags  $F_{\bullet}^{\gamma,\sigma}$ where $\sigma\in {\mathfrak S}_{p+q}$
is $\gamma$-shuffled.
\end{proposition}

\begin{proposition}[${\sf S}$-fixed points of $\bem$ for $({\sf GL}_{p+q},{\sf GL}_p\times {\sf GL}_q)$]
\label{prop:fixedpoints727}
The set of ${\sf S}$-fixed points of $\bem$ correspond to 
$P$-growth diagrams whose initial vertical chain is $F_{\bullet}^{\gamma,\sigma}$ (where $Y=Y_\gamma$). 
\end{proposition}
\begin{proof}
The following is straightforward:

\begin{claim}
\label{lemma:unique}
Fix a coordinate flag $F_{\bullet}$ for the initial vertical chain, i.e.~$F_{\bullet}\in ({\sf G}/{\sf B})^{\sf T}$.  There exists exactly one $P$-growth of $F_{\bullet}$ which uses only subspaces of the form $\langle \vec e_{i_1},\ldots\vec e_{i_d}\rangle$.
\end{claim}

Clearly any such $P$-growth of $F_{\bullet}$ is an ${\sf S}$-fixed point of $\bem$. Conversely, consider any ${\sf S}$-fixed point of $\bem$ and its corresponding diagram. The left border is an ${\sf S}$-fixed point
of $Y=Y_\gamma$. The result then holds by Proposition~\ref{prop:729yama} together with Claim~\ref{lemma:unique}.
\end{proof}

Proposition~\ref{prop:thefixed} gives a general form of Proposition~\ref{prop:fixedpoints727}.

\begin{corollary}
\label{cor:fixedpointcount}
$\#\left( \bem \right)^{\sf S}=p!q!2^{|Q|}$.
\end{corollary}

Similar descriptions for the torus fixed points can be given for the other two symmetric
pairs of the form $({\sf GL}_n, {\sf K})$. In these cases ${\sf T}\neq {\sf S}$, however it is known that the fixed points in the respective flag varieties agree (see \cite[pg.~128]{Brion1999}). In brief,  
in the case 
$({\sf G},{\sf K})=({\sf GL}_{2n},{\sf Sp}_{2n})$, as elements of ${\mathfrak S}_{2n}$, these 
${\sf S}$-fixed points correspond to ``mirrored'' permutations, i.e. 
those permutations $w$ having the property that 
$w(2n+1-i) = 2n+1-w(i)$ 
for each $i$; this is described in detail in \cite{Wyser-13}.  Similarly, 
in the case of 
$({\sf G},{\sf K})=({\sf GL}_n,{\sf O}_n)$, these fixed points correspond to mirrored elements of of ${\mathfrak S}_n$, as described in \cite{Wyser-13}.

In \cite{EPTY-16+} one considers Bott-Samelson varieties in relation to zonotopal tilings of an \emph{Eltnitsky} polygon. This puts a poset structure
on Bott-Samelson varieties (in type $A$) by introducing 
generalized Bott-Samelson varieties for which the 
fibers are larger flag varieties rather than ${\mathbb P}^1$'s. The diagram definition of $\bem$ permits one to obtain similar definitions and results here \emph{mutatis mutandis}.

\section{Moment polytopes for the ${\sf GL}_p\times {\sf GL}_q$ case}\label{sec:furtherPQ}

Recall from the end of Section~\ref{sec:mps} that the BEM polytope ${\mathcal P}_{Y_0,Q}$ denotes the moment polytope $\Phi(\mathcal{BEM}^{Y_0,Q})$.

\begin{example}\label{ex:runningpoly} Let
$Q=(3,2)$  and $Y_0=Y_{++--}$. 
Following the construction in Section~\ref{sec:A-BEMs}, 
and applying Theorem~\ref{thm:wyserbruhat},
$\bem$ is described by the following diagram.

\begin{figure}[h]\centering
  \begin{tikzpicture}
\node (top) at (0,0) {$\C^4$};
    \node [below of=top] (flag3)  {$W_3$};
    \node [below of=flag3] (flag2) {$E_2$};
    \node [below of=flag2] (flag1) {$W_1$};
    \node [right of=flag2] (L1) {$V_2$};
        \node [right of=flag3] (P1) {$V_1$};
\node [below of=flag1] (down) {$\C^0$};
\draw [thick,shorten <=-2pt] (top) -- (flag3);
\draw [thick,shorten <=-2pt] (flag3) -- (flag2);
\draw [thick,shorten <=-2pt] (flag2) -- (flag1);
\draw [thick,shorten <=-2pt] (P1) -- (L1);
\draw [thick,shorten <=-2pt] (top) -- (P1);
\draw [thick,shorten <=-2pt] (flag3) -- (L1);
\draw [thick,shorten <=-2pt] (flag1) -- (down);
\draw [thick,shorten <=-2pt] (L1) -- (flag1);
\end{tikzpicture}
\end{figure}

By Corollary~\ref{cor:fixedpointcount}, $\bem$ has $4\cdot 2^2=16$ ${\sf S}$-fixed points. We apply Theorem~\ref{thm:moment} to construct the moment polytope.
First, by \eqref{eq:bsmoment}, $\Phi(BS^Q)$ is the convex hull of the following points:
\begin{center}
\begin{center}
\begin{tabular}{ |c|c| } 
 \hline
 $({\gb_1},{\gb_2})$ & $\displaystyle{\sum_{i=1}^r\omega_i+\sum_{i={|Q|}}^{1}s_{\beta_{|Q|}}\cdots s_{\beta_i} \omega_i}$  \\ 
  \hline
 $(-,-)$ & $(5,4,2,0)=(3,2,1,0)+s_-\cdot(1,1,0,0)+s_-s_-\cdot(1,1,1,0)$ \\ 
 $(3,-)$ & $(5,4,1,1)=(3,2,1,0)+s_-\cdot(1,1,0,0)+s_-s_3\cdot(1,1,1,0)\ $  \\
$(-,2)$ & $(5,3,3,0)=(3,2,1,0)+s_2\cdot(1,1,0,0)\ +s_2s_-\cdot(1,1,1,0)\ $  \\
$(3,2)$ & $(5,2,3,1)=(3,2,1,0)+s_2\cdot(1,1,0,0)\ +s_2s_3\cdot(1,1,1,0)\ \ $  \\
 \hline
\end{tabular}
\end{center}
\end{center}
The polytope $\Phi(BS^Q)$ is the white quadrilateral in Figure~\ref{fig:1}. We consider the reflections of $\Phi(BS^Q)$
 by the ${\sf T}$-fixed points of $Y_0$, corresponding to the
$++--$ shuffled permutations:
\[[1,2,3,4], \ [2,1,3,4], \ [1,2,4,3], \text{ and $[2,1,4,3]$.}\] 
By Theorem~\ref{thm:moment}, ${\mathcal P}_{Y_0,Q}$ is the convex hull of the following reflections
\begin{align*}[1,2,3,4]\cdot\Phi(BS^Q)=&{\sf conv}\{(5,4,2,0),(5,4,1,1),(5,3,3,0),(5,2,3,1)\},\\
[2,1,3,4]\cdot\Phi(BS^Q)=&{\sf conv}\{(4,5,2,0),(4,5,1,1),(3,5,3,0),(2,5,3,1)\},\\
[1,2,4,3]\cdot\Phi(BS^Q)=&{\sf conv}\{(5,4,0,2),(5,4,1,1),(5,3,0,3),(5,2,1,3)\}, \text{ and}\\
[2,1,4,3]\cdot\Phi(BS^Q)=&{\sf conv}\{(4,5,0,2),(4,5,1,1),(3,5,0,3),(2,5,1,3)\}.
\end{align*}

\begin{figure}[h]\centering
\includegraphics{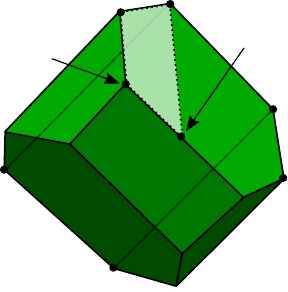} 
\put(-70,144){{\footnotesize$id,(3,2)$}}
\put(-157,110){{\footnotesize$id,(-,-)$}}
\put(-25,120){{\footnotesize$id,(3,-)$}}
\put(-127,135){{\footnotesize$id,(-,2)$}}
\put(3,52){{\footnotesize$[1,2,4,3],(-,2)$}}
\put(-2,85){{\footnotesize$[1,2,4,3],(3,2)$}}
\put(-207,55){{\footnotesize$[2,1,3,4],(3,2)$}}
\put(-154,5){{\footnotesize$[2,1,4,3],(3,2)$}}
\caption{\label{fig:1} ${\mathcal P}_{Y_0,Q}$ for $Y_0=Y_{++--}$ and 
$Q=(3,2)$ is the convex hull of four reflections in ${\mathbb R}^3$ of the Bott-Samelson polytope (white). We have labelled some of the points $\Phi(p_{(x,J)})$ using $x,J$; all other points can be inferred from these.}
\end{figure}
\end{example}

By the discussion of Section~\ref{sec_typeA_BEM}, the number of choices of closed orbits $Y_0=Y_{\gamma}$ equals the number of matchless clans in ${\tt Clans}_{p,q}$ and this number is  ${p+q\choose p}$.
However, as we verify in the following theorem, if we fix $Q$ and let $Y_\gamma$ vary, then all the BEM-polytopes are isometric, being reflections of one other.

\begin{theorem}[Reduction to $+\cdots +-\cdots -$ case]\label{prop:theone} 
${\mathcal P}_{Y_{\gamma},Q}$ 
is a $w$-reflection of ${\mathcal P}_{Y_{+\ldots+-\ldots-},Q}$ where $w$ is the smallest permutation such that 
$w\cdot (+\ldots+-\ldots-)=\gamma$.
\end{theorem}

\begin{proof} 
Suppose $\gamma\in {\tt Clans}_{p,q}$ is matchless and there
exists an $i$ such that $\gamma_i=-$ and $\gamma_{i+1}=+$.
Let $\gamma'\in {\tt Clans}_{p,q}$ be obtained by interchanging
$-+\mapsto +-$ at those positions.

By Proposition~\ref{prop:729yama}, the ${\sf T}$-fixed points 
of $Y_{\gamma}$ are the $\gamma$-shuffled permutations; call this set ${\mathcal A}$. Similarly, the ${\sf T}$-fixed points of $Y_{\gamma'}$
are the $\gamma'$-shuffled permutations; call this set ${\mathcal B}$. 

\begin{claim}
${\mathcal A}s_i={\mathcal B}$.
\end{claim}
\noindent
\emph{Proof of claim:}
Let $\sigma\in {\mathcal A}$. Since $\gamma_i=-$, by definition 
$\sigma(i)\in \{p+1,p+2,\ldots,n\}$. Also, since $\gamma_{i+1}=+$,
$\sigma(i+1)\in \{1,2,\ldots,p\}$. Thus if $\sigma'=\sigma s_i$ then
$\sigma'(i)\in \{1,2,\ldots,p\}$ and $\sigma'(i+1)\in \{p+1,p+2,\ldots,n\}$, as is required for $\sigma'\in {\mathcal B}$. The claim follows.
\qed

The claim, combined with Proposition~\ref{prop:fixedpoints727}
imply that the ${\sf T}$-fixed points of $\mathcal{BEM}^{Y_{\gamma'},Q}$ are the $s_i$-reflection of those of 
$\mathcal{BEM}^{Y_\gamma,Q}$. Since the moment map images are determined by these ${\sf T}$-fixed points, the respective
polytopes must be an $s_i$ reflection of one another.
Now iterate this process down to 
the case $+\cdots +-\cdots -$.
\end{proof}

The Table~\ref{table:someinfo} summarizes some information about the resulting polytopes for $p=q=2$. In view of Theorem \ref{prop:theone}, we only need to consider $\gamma=++--$. We have restricted to $Q$ reduced and $|Q|\leq 3$ for brevity.
Actually, based on such calculations, it seems true that if $Q$ and $Q'$ are Demazure words for the same $w$ then the BEM polytopes
are combinatorially equivalent. For example, $Q=(1),(1,1),(1,1,1)$ are all two dimensional with $(V,E,F)=(4,4,1)$. However, we have no proof of this at present.

\begin{table}[t]

\begin{tabular}{ |c|c|c|c|c|c| }
 \hline 
$Q$ & dim & $V$ & $E$ & $F$  \\ 
\hline\hline
$(1)$ & $2$ & $4$ & $4$ & $1$\\

\hline 
$(2)$ & $3$ & $8$ &	$12$ &	$6$ \\ 

\hline
$(3)$ & $2$ & $4$ & $4$ & $1$\\
 \hline\hline
 
 $(1,2)$ & $3$ & $12$ & $18$ & $8$\\
 \hline

 $(1,3)\equiv (3,1)$ & $2$ & $4$ & $4$ & $1$\\
 \hline

 $(2,1)$ & $3$ & $8$ & $12$ & $6$\\
 \hline
 
  $(2,3)$ & $3$ & $8$ & $12$ & $6$\\
 \hline
  $(3,2)$ & $3$ & $12$ & $18$ & $8$\\
  
 \hline\hline
 
  $(1,2,1)\equiv(2,1,2)$ & $3$ & $12$ & $18$ & $8$\\
  
 \hline
   $(1,2,3)$ & $3$ & $12$ & $18$ & $8$\\
 \hline
  $(1,3,2)\equiv (3,1,2)$ & $3$ & $8$ & $12$ & $6$\\
\hline

   $(2,1,3)\equiv (2,3,1)$ & $3$ & $8$ & $12$ & $6$\\
 \hline 
 
    $(2,3,2)$ & $3$ & $12$ & $18$ & $8$\\
     \hline  
    $(3,2,1)$ & $3$ &$12$ & $18$ & $8$\\
 \hline   
     $(3,2,3)$ & $3$ &$12$ & $18$ & $8$\\
 \hline
\end{tabular}
\caption{\label{table:someinfo} BEM polytope data for $({\sf GL}_4, {\sf GL_2}\times {\sf GL}_2)$ where $Q$ is reduced and $|Q|\leq 3$}
\end{table}

Following \cite[Section 5]{Tymo}, let $X$ be a projective algebraic variety with a torus action ${\sf T}$. Suppose
$p\in X^{\sf T}$. 
Let $T_p(X)$ be the tangent space; this too carries a ${\sf T}$ action
and a ${\sf T}_{\R}$ action.
The ${\sf T}_{\R}$-decomposition is
$T_p(X)=\bigoplus_{\alpha} E_{\alpha}$,
where $E_{\alpha}$ are the eigenspaces with 
eigenvalues $\alpha\in {\mathfrak t}^*$. These $\{\alpha\}$ are the
{\bf ${\sf T}$-weights}. 
The nonnegative cone spanned by these ${\sf T}$-weights of $T_p(X)$
is equal to 
the cone spanned by the edges of the moment polytope $\Phi(X)$ 
incident to $\Phi(p)$.

For $w=s_{j_1}s_{j_2}\cdots s_{j_\ell}$ a reduced expression of $w$ we define $$\text{inv}(w):=\{\alpha_{j_1},s_{j_1}(\alpha_{j_2}),\ldots,s_{j_1}s_{j_2}\cdots s_{j_{\ell-1}}(\alpha_{j_\ell})\}.$$

\begin{theorem}[Combinatorial description of ${\sf T}$-weights]
\label{prop:tweights}
Let $Q=(j_1,j_2,\ldots,j_N)$ be a word and $J=(\beta_1,\ldots,\beta_N)$ be a subword of $Q$. 
The ${\sf T}$-weights of the tangent space of $\mathcal{BEM}^{Y_{+\ldots+-\ldots-},Q}$ at $p_{(u{\sf B}, J)}$, where $u{\sf B}$ is a ${\sf T}$-fixed point of $Y_{+\ldots+-\ldots-}$, are
$$u\cdot(-\text{inv}(w))\cup u\cdot\{s_{\beta_N}\cdot(-\alpha_{j_N}),s_{\beta_N}s_{\beta_{N-1}}\cdot(-\alpha_{j_{N-1}}),\ldots,s_{\beta_N}\cdots s_{\beta_1}\cdot(-\alpha_{j_1})\},
$$
where $w=[p,p-1,\ldots,1,n,n-1,\ldots,p+1]$.
\end{theorem}
\begin{proof}
We apply:
 \begin{theorem}\cite[Corollary 3.11]{Graham-Zierau}\label{thm:GZ} Let $Q_0,\ldots,Q_n$ be subgroups of an algebraic group $G$ and let $T$ be a torus in $G$. Suppose that $R_0,\ldots,R_n$ are subgroups of $G$ with $R_i\subset Q_{i-1}\cap Q_i$ for $i>0$ and $R_0\subset Q_0$. Let 
\[X=Q_n\times^{R_n}Q_{n-1}\times^{R_{n-1}}\cdots\times^{R_2}Q_1\times^{R_1}Q_0/R_0\] 
and $[q_n,\ldots,q_0]\in X$ a $T$-fixed point. Assume in addition that for every $i$, $q_i^{-1}\cdots q_n^{-1}$ is in the normalizer of $T$. Then the weights of $T$ acting on the tangent space $T_{[q_n,\ldots,q_0]}X$ is the multiset union of
 $$
 q_nq_{n-1}\cdots q_i \cdot \{\text{weights of } T \text{ acting on } Q_i/R_i\}
 $$
 where $i$ runs from $n$ to $0$.
 \end{theorem}
 
 More precisely, we apply this result to $\sf T$ and
\[{\mathcal B}{\mathcal E}^{Y_0,Q} = \widetilde{Y_0} \times^{\sf B} {\sf P}_{{j_N}} \times^{\sf B} {\sf P}_{{j_{N-1}}} \times^{\sf B} \hdots \times^{\sf B} {\sf P}_{{j_1}}/{\sf B},\] where $\widetilde{Y_0}$ is the preimage of $Y_0=Y_{+\ldots+-\ldots-}$ in ${\sf G}$ under ${\sf G}\rightarrow{\sf G}/{\sf B}$.

Let us verify that ${\mathcal B}{\mathcal E}^{Y_0,Q}$ satisfies the required hypotheses.
The orbit 
$$Y_{+\ldots+-\ldots-}=\{\mathbb{C}^0\subset F_1\subset \cdots\subset F_{p-1}\subset \mathbb{C}^p\subset F_{p+1}\subset \cdots\subset \mathbb{C}^{p+q}\}.
$$
Thus $\widetilde{Y_0}$ is the maximal parabolic subgroup $P_{\widehat p}$.
We have that $\widetilde{Y_0}, {\sf P}_{{j_N}},\hdots,  {\sf P}_{{j_1}}$ are subgroups of ${\sf GL}_n$.
Since ${\sf B}$ is a Borel subgroup then ${\sf B}\subset 
\widetilde{Y_0}\cap{\sf P}_{j_{N}}$
and 
${\sf B}\subset {\sf P}_{j_{t-1}}\cap{\sf P}_{j_{t}}$ 
for $1\leq t\leq N$.

The ${\sf T}$-fixed point of ${\mathcal B}{\mathcal E}^{Y_0,Q}$ corresponding to $p_{(u{\sf B},J)}$ is $[u,s_{\beta_{N}},s_{\beta_{N-1}},\ldots,s_{\gb_1}{\sf B}]$, where $u\in N({\sf T})$. Therefore $(us_{\beta_{N}}s_{\beta_{N-1}}\ldots s_{\gb_i})^{-1}$ is in the normalizer of ${\sf T}$ for all $i$.
We have now verified that ${\mathcal B}{\mathcal E}^{Y_0,Q}$ satisfies the required hypotheses. 

Since $Y_0$ is the Schubert variety for $w$, the ${\sf T}$-weights of $\widetilde{Y_0}/{\sf B}=Y_0$ at ${\sf B}$ are the negatives of the inversions of $w$.
The ${\sf T}$-weight of ${\sf P}_{\ga_i}/{\sf B}$ at ${\sf B}$ is the simple root $\ga_i$.
By Theorem~\ref{thm:GZ}, the ${\sf T}$-weights of ${\mathcal B}{\mathcal E}^{Y_0,Q}$ at the fixed point $[u,s_{\beta_{N}},s_{\beta_{N-1}},\ldots,s_{\gb_1}]$ is the following multiset-union
\begin{equation*} u\cdot(-\text{inv}(w))\cup \{us_{\beta_{N}}\cdot( -\ga_{j_{N}})\}\cup  \{us_{\beta_{N}}s_{\beta_{N-1}}\cdot (-\ga_{j_{N-1}})\}\cup \cdots  \{us_{\beta_{N}}\cdots s_{\beta_1}\cdot (-\ga_{j_{1}})\}
\end{equation*}
By Theorem~\ref{thm:main}, the ${\sf T}$-weights for the tangent spaces of ${\mathcal B}{\mathcal E}^{Y_0,Q}$ are the same as those for $\bem$.
\end{proof}

\begin{corollary}\label{cor:vertexpartial} The point $\Phi(p_{(u{\sf B},J)})$ 
is a vertex of ${\mathcal P}_{Y_{+\cdots+-\cdots -},Q}$ if and only if $\Phi(p_{({\sf B},J)})$ is a vertex of ${\mathcal P}_{Y_{+\cdots+-\cdots -},Q}$.
\end{corollary}
\begin{proof}
$\Phi(p_{(w{\sf B},J)})$ is a vertex whenever there is not a line in the cone spanned by the ${\sf T}$-weights of the tangent space 
$T_{p_{(w{\sf B},J)}}(\mathcal{BEM}^{Y_{+\cdots+-\cdots -},Q})$. 
By Theorem~\ref{prop:tweights}, 
\[\text{${\sf T}$-weights of $T_{p_{(u{\sf B},J)}}(\mathcal{BEM}^{Y_{+\cdots+-\cdots -},Q})$}
=
u\cdot (\text{${\sf T}$-weights of $T_{p_{({\sf B},J)}}(\mathcal{BEM}^{Y_{+\cdots+-\cdots -},Q})$}).
\]
The claim follows since a
cone contains a line if and only any reflection contains a line.
\end{proof}

\begin{example} Consider the BEM polytope ${\mathcal P}_{Y_0,Q}$ of Example~\ref{ex:runningpoly} and Figure~\ref{fig:1}. 
The vertices of ${\mathcal P}_{Y_0,Q}$ adjacent to $\Phi(p_{({\sf B},(3,2))})$ are 
\[\Phi(p_{({\sf B},(-,2))}), \ \ \Phi(p_{([1,2,4,3]{\sf B},(3,2))}), \ \ \  \text{and \ \ \ $\Phi(p_{([2,1,3,4]{\sf B},(3,2))})$.}\]
The cone spanned by the edges of ${\mathcal P}_{Y_0,Q}$ incident to $\Phi(p_{({\sf B},(3,2))})$ is
\begin{align*}
{\sf pos}&\{\Phi(p_{({\sf B},(-,2))})-\Phi(p_{({\sf B},(3,2))}),\Phi(p_{(s_3{\sf B},(3,2))})-\Phi(p_{({\sf B},(3,2))}),\Phi(p_{([s_1{\sf B},(3,2))})-\Phi(p_{({\sf B},(3,2))})\}\\
&={\sf pos}\{(0,0,-1,1),(0,1,0,-1),(-1,1,0,0)\}.
\end{align*}
Let us compute the ${\sf T}$-weights for the tangent space of $\bem$ at $p_{({\sf B},(3,2))}$. We have that $w=[2,1,4,3]=s_1s_3$ so
$$
\text{inv}(w)=\{\ga_1,s_1(\ga_3)\}=\{\ga_1,\ga_3\}=\{(1,-1,0,0),(0,0,1,-1)\}.
$$
Since $J=(3,2)$, then by Theorem~\ref{prop:tweights} the $\sf T$-weights are
\begin{align*}\{-\ga_1,-\ga_3\}\cup\{s_{2}(-\alpha_{2}),s_{2}s_{3}(-\alpha_3)\}&=\{-\ga_1,-\ga_3,\ga_2,\ga_2+\ga_3\}\\
&=\{(-1,1,0,0),(0,0,-1,1),(0,1,-1,0),(0,1,0,-1)\}.
\end{align*}
The cone spanned by the $\sf T$-weights coincides with the cone spanned by the edges incident to $\Phi(p_{({\sf B},(3,2)})$.
Since this cone does not contain a line it follows that $\Phi(p_{({\sf B},(3,2)})$ is a vertex of ${\mathcal P}_{Y_0,Q}$.

Now consider the $\sf T$-fixed point $p_{({\sf B},(3,-))}$. 
The $\sf T$-weights for the tangent space of $\bem$ at  $p_{({\sf B},(3,-))}$ are
\begin{align*}\{-\ga_1,-\ga_3\}\cup\{s_{-}(-\alpha_{2}),s_{-}s_{3}(-\alpha_3)\}&=\{-\ga_1,-\ga_3,-\ga_2,\ga_3\}
\\
&=\{(-1,1,0,0),(0,0,-1,1,),(0,-1,1,0),(0,0,1,-1)\}.
\end{align*}
By Theorem~\ref{prop:tweights} the cone spanned by these vectors is the cone spanned by the edges incident to $\phi(p_{({\sf B},(3,-))})$. Since this cone contains the line spanned by $\ga_3$ then this point is not a vertex of ${\mathcal P}_{Y_0,Q}$.\qed
\end{example}

Although we have not done so here, it should be possible 
to give a combinatorial description of the vertices of ${\mathcal P}_{Y_0,Q}$.
Doing so is equivalent to classifying the $\sf T$-fixed points for which the cone spanned by the $\sf T$-weights does not contain a line. 

We conclude this paper with:

\begin{theorem}[Dimension of ${\mathcal P}_{Y_0,Q}$]\label{thm:dimaug2}
For $({\sf G},{\sf K})=({\sf GL}_{p+q},{\sf GL}_p\times {\sf GL}_q)$,
$$
\dim({\mathcal P}_{Y_0,Q})=\begin{cases} p+q-1,\ \text{ if }\ p \text{ is in } Q, \text{ and}\\
 p+q-2,\ \text{ if }\ p \text{ is not in } Q.\end{cases}
$$
\end{theorem}
\begin{proof}
A $T$-action on a space $X$ is \emph{effective} if each element of $T$, other than the identity, moves at least one point of $X$.
In the proof of \cite[Corollary 27.2]{CannasdaSilva} it is shown that for an effective Hamiltonian $T$-action the dimension of the corresponding moment polytope equals the dimension of the torus.
If the $T$-action is not effective it is known that it can be reduced it to an effective action with the same moment polytope. The \emph{stabilizer} of the $T$-action is the normal subgroup
$$
S_T(X):=\{t\in T \mid t\cdot x=x \text{ for all } x\in X\}.
$$
The $T$-action on $X$ reduces to the effective action of $T/S_T(X)$ given by $tS_T(X)\cdot x:=t\cdot x$.
To prove the Theorem we will consider the cases in which $p$ is in $Q$ and when it isn't separately.
In each case we will explicitly compute the stabilizer of the $\sf T$-action on $\bem$.
First, note that 
	$
	{\sf T}_1=\{{\sf t}\in {\sf T}\mid t_1=\ldots =t_{n}\}
	$
acts trivially on $\bem$ and therefore it is a subgroup of $S_{\sf T}(\bem)$.

In view of Theorem~\ref{prop:theone}, we assume without loss of generality, that $Y_0=Y_{+\cdots+-\cdots -}$, i.e.
	\begin{equation*}\label{eq:ysippmm}
	Y_0=\{(F_1,F_2,\ldots,F_{n-1})\in{\sf G}/{\sf B} \mid F_p=E_p\}.
	\end{equation*}
Consider the projection $\varphi_0:\bem\to Y_0$ that maps a point in the configuration space to
the leftmost flag (corresponding to the leftmost border) in the diagram.
Since the projection is $\sf T$-equivariant, $S_{\sf T}(\bem)$ is a subgroup of $S_{\sf T}(Y_0)$.
Let us first describe $S_{\sf T}(Y_0)$.
Note that 
	$$
	{\sf T}_2=\{{\sf t}\in {\sf T}\mid t_1=\ldots =t_p,\ t_{p+1}=\ldots=t_{n}\}
	$$
acts trivially on $Y_0$. 
Furthermore, since $Y_0$ is isomorphic to ${\sf GL}_p/{\sf B}\times {\sf GL}_q/{\sf B}$ then $S_{\sf T}(Y_0)$ is isomorphic to $S_{\sf T}({\sf GL}_p/{\sf B})\times S_{\sf T}({\sf GL}_q/{\sf B})$. 
It follows that $S_{\sf T}(Y_0)={\sf T}_2$ and therefore $S_{\sf T}(\bem)$ is either equal to ${\sf T}_1$ or ${\sf T}_2$.

First suppose that $p$ is in $Q$ and let $k$ be such that $j_k=p$, where $Q=(j_1,\ldots,j_N)$. 
It is straightforward from Section~\ref{sec:A-BEMs} that if we set 
	$
	V_k={\rm span}\{\vec e_1,\ldots,\vec e_{p-1},\vec e_{p}+\vec e_{p+1}\},
	$
then the following holds:
	$$
	(E_1,E_2,\ldots,E_{n-1},E_{j_1},E_{j_2},\ldots E_{j_{k-1}},V_k,E_{j_{k+1}},\ldots, E_{j_N})\in \bem.
	$$
Furthermore, for ${\sf t}\in {\sf T}_2$ 
\begin{align*}
{\sf t}\cdot V_k
&={\rm span}\{\vec e_1,\vec e_2,\ldots,\vec e_{p-1},\vec e_{p}+t_{n}\vec e_{p+1}\} \neq V_{k},
\end{align*}
so $ {\sf T}_2$ does not act trivially on $\bem$ and $S_{\sf T}(\bem)={\sf T}_1$.
From this and the first paragraph it follows that $\dim({\mathcal P}_{Y_0,Q})=p+q-1$. 

Now suppose that $p$ is not in $Q$.
This implies that every vector space appearing in 
	$$
	(F_1,F_2,\ldots,F_{n-1},V_1,V_2,\ldots V_N)\in \bem
	$$ 
must either be a subspace of $E_p$ or contain $E_p$.
So ${\sf T}_2$ acts trivially on $\bem$ and 
$S_{\sf T}(\bem)={\sf T}_2$.
We conclude that $\dim({\mathcal P}_{Y_0,Q})=p+q-2$.
\end{proof}

\begin{example}
The data of Table~\ref{table:someinfo} is consistent with
Theorem~\ref{thm:dimaug2}. Furthermore,
note that the dimension characterization 
only depends on $p$ 
and not $q$. Indeed, if $p=2$ and $q=3$, one can check ${\mathcal P}_{Y_{++---},(3)}$ has 
dimension $3$, also in agreement with the theorem.\qed
\end{example}

\section*{Acknowledgements}
We thank D.~Barbasch, B.~Elek, S.~Evens, W.~Graham, A.~Knutson, E.~Lerman, L.~Li, J.~Pascaleff, and A.~Woo, and  for helpful conversations. We thank S.~Evens for clarifying and correcting our understanding
of the hypotheses of \cite[Proposition~6.4]{Barbasch-Evens-94}. We also thank anonymous referees for their suggestions. LE was supported by an AMS-Simons travel grant, NSF grant DMS 1855598 and NSF CAREER grant DMS 2142656. AY was 
supported by NSF grant  DMS 1500691, an NSF RTG in combinatorics DMS 1937241, and 
Simons Collaboration Grant 582242.

\bibliography{KBSbiblio}
\bibliographystyle{plain}

\end{document}